\UseRawInputEncoding
\documentclass[12pt, leqno]{amsart}

\usepackage[OT2,T1]{fontenc}
\usepackage{amstext}
\usepackage{amsthm}
\usepackage{amsopn}
\usepackage{amsfonts}
\usepackage{amsmath}
\usepackage{latexsym}
\usepackage[francais,english]{babel}
\usepackage{amscd}
\usepackage{amssymb}
\usepackage{amsmath}
\usepackage[all,cmtip]{xy}
\usepackage{leftidx}
\usepackage{graphicx}
\usepackage{etoolbox}
\patchcmd{\section}{\normalfont\scshape\centering}{\normalfont\bfseries}{}{}
\patchcmd{\subsection}{-.5em}{.5em}{}{}
\renewenvironment{proof}{{\noindent\bfseries Proof.}}{}

\newtheorem{theo}{{Theorem}}[section]
\newtheorem{coro}[theo]{{Corollary}}
\newtheorem{lemma}[theo]{{Lemma}}
\newtheorem{prop}[theo]{Proposition}

\theoremstyle{definition}
\newtheorem{remark}[theo]{\textbf{Remark}}

\newtheorem{example}[theo]{Example}
\numberwithin{equation}{section}

\newcommand{\ra}{\rightarrow}

\newcommand{\ol}{\overline}

\newcommand{\rI}{\mathrm{I}}

\newcommand{\rR}{\mathrm{R}}

\newcommand{\cI}{\mathcal{I}}
\newcommand{\cE}{\mathcal{E}}

\newcommand{\cP}{\mathcal{P}}

\newcommand{\Br}{\mathrm{Br}}

\newcommand{\Gal}{\mathrm{Gal}}

\newcommand{\inv}{\mathrm{inv}}
\newcommand{\ind}{\mathrm{Ind}}

\newcommand{\gm}{\mathbb{G}}

\newcommand{\rat}{\mathbb{Q}}
\newcommand{\ent}{\mathbb{Z}}

\DeclareFontEncoding{OT2}{}{} % to enable usage of cyrillic fonts
  \newcommand{\textcyr}[1]{%
    {\fontencoding{OT2}\fontfamily{wncyr}\fontseries{m}\fontshape{n}%
     \selectfont #1}}
\newcommand{\sha}{{\mbox{\textcyr{Sh}}}}

\begin{document}
\tolerance 400 \pretolerance 200 \selectlanguage{english}

\title{Hasse principles for multinorm equations}
\author{E. Bayer-Fluckiger}
\author{T.-Y. Lee}
\author{R. Parimala}
\date{\today}
\maketitle

\begin{abstract}
A classical result of Hasse states that  the norm principle holds for finite cyclic extensions of global fields, in other words local norms are
global norms. We investigate the norm principle for finite dimensional commutative \'etale algebras over global fields; since such an algebra is a product of
separable extensions, this is often called the multinorm principle. Under the assumption that the \'etale
algebra contains a cyclic factor, we give a necessary and sufficient condition for  the Hasse principle to hold, in terms of an explicitly constructed element of  a finite abelian group.
This can be seen as an explicit description of the Brauer-Manin obstruction to the Hasse principle.
%We also give a necessary and sufficient condition for the Hasse principle to hold when one of the factors is metacyclic.

\medskip

\noindent {\em Keywords:} multinorm, tori, Shafarevich groups, local-global principles, Brauer-Manin obstructions

\medskip

\noindent {\em MSC 2000:} 11G35, 12G05,14G05, 14G25.
\end{abstract}

\small{} \normalsize

\medskip

\selectlanguage{english}
\section{Introduction}

\bigskip

Let $k$ be a global field,  and $L$ be a finite dimensional commutative \'etale algebra over  $k$.
We say that the {\it Hasse norm principle} holds for $L$ if the local-global principle holds for the equation
\begin{equation}\label{e0}
N_{L/k}(t)=c
\end{equation}
for all $c\in k^\times$; this terminology is inspired by  Hasse's result that the norm principle holds in the case of {\it cyclic extensions}
(\cite{Ha1}, \cite{Ha2} \S I (3.11) and \S II (15)).
Over the years, the norm principle for separable field extensions attracted a lot of attention; it is known not to hold in general, and many positive results are also available, see for instance \cite{PlR}, pages 308-309 for a survey; for more recent results,  see  \cite{BN}, \cite{FLN}, and
the references therein.

\medskip

It is natural to ask for Hasse principles in the case when $L$ is a finite dimensional commutative {\it \'etale algebra}, and not just a field extension.
Since $L$ is by definition a product of separable extensions, the equation (\ref{e0}) is often called a multinorm equation.

\medskip
This more general problem was also studied extensively, in particular by H\"urlimann (\cite {H}), Colliot-Th\'el\`ene and Sansuc (unpublished), Platonov and Rapinchuk (see \cite {PlR}, sections 6.3  and 9.3), Prasad and Rapinchuk (\cite {PR1}, Section 4),
Pollio and Rapinchuk (\cite{PoR}),  Demarche and Wei (\cite{DW}), Pollio  (\cite {Po}). Multinorm equations also arise when dealing with classical groups
of type $A_n$ (see for instance \cite{PR1} Prop. 4.2).

\medskip

In spite of many interesting results, some quite simple cases were still open. We illustrate this, as well as our results, by the following example:

\medskip
\noindent
{\bf Example.} Assume that $L$ is a product of $n$ non-isomorphic quadratic field extensions of $k$. If $n = 1$ or $n = 2$, then the Hasse principle holds for $L$ - this is clear for $n = 1$,
and easy for $n = 2$ (for instance, it is a consequence of \cite {H}, Proposition 3.3). It is also well-known that it does not hold in general when $n = 3$ (see
for instance [CT 14]). In the present paper, we show that  {\it the Hasse principle holds if $n \ge 4$.}

\medskip

To obtain this result and others, let us assume that one of the factors of $L$ is  a {\it cyclic field extension} of $k$.  Under this hypothesis, we
construct a finite abelian group $\sha(L)$ having the property that

\medskip
\centerline {$\sha(L) = 0 \iff$ the Hasse principle holds for $L$}

\medskip
\noindent
(cf.
Section $5$). Assume now that $\sha(L) \not = 0$, and that $c \in k^{\times}$ is such that (\ref{e0}) has a solution locally everywhere. Then we
construct a homomorphism $$\alpha_c : \sha(L) \to \rat / \ent$$ such that

\medskip
\centerline {(\ref{e0})  has a solution over $k \iff \alpha_c =0$}

\medskip
\noindent (see Sections 6 and 7, in particular Theorem \ref{main_theo}).
%******************

\medskip These results can be summarized as follows : let $I_L$ be the id\`ele group of $L$. Then sending $c \in k^{\times}$ to $\alpha_c$ gives rise to an isomorphism $$k^{\times} \cap N_{L/k}(I_L) / N_{L/k}(L^{\times}) \to \sha(L)^*$$ (where $\sha(L)^*$ is the dual of $\sha(L)$, cf. Corollary \ref{idele}).

\medskip We also give a necessary and sufficient condition for the Hasse principle to hold when one of the factors is metacyclic (see Proposition \ref{metacyclic}).

\medskip The results are easy to use.  To illustrate this, we consider the case where
$L$ is a {\it product of cyclic extensions}; assume that $L = \underset{i \in J} \prod K_i $, where $K_i/k$ is a cyclic extension of degree $d_i$.
Let $\cP$ be the set of prime numbers dividing $\underset{i \in J} \prod d_i$. For all $p \in \cP $ and all $i \in J$, let $K_i(p)$ be the largest subfield
of $K_i$ such that $[K_i(p):k]$ is a power of $p$, and set $L(p) = \underset{i \in J} \prod K_i(p)$.
Then we have
$$\sha (L) =  \underset{p \in \cP}{\oplus}  \sha(L(p)),$$
(see Proposition \ref{prime decomposition}).

\medskip
For any cyclic field extension $K/k$ of prime power degree, we denote by $K_{\rm prim}$  the unique subfield of $K$ of degree $p$ over $k$.  Set $$L(p)_{\rm prim} = \underset{i \in J} \prod K_i(p)_{\rm prim}.$$
Then we have

$$\sha(L) = 0 \iff  \underset{p \in \cP(L)}{\oplus}   \sha(L(p)_{\rm prim}) = 0,$$
(cf Theorem \ref{zero}), and $$\sha(L(p)_{prim}) \simeq (\ent / p \ent)^{m_p(L)},$$
where $\cP(L)$ is a set of prime numbers (subset of $\cP$), and $m_p(L)$ is a positive integer; both are determined explicitly (see Theorem \ref {cyclicproduct}).

\medskip
The paper is structured as follows. Sections 1-4 contain some preliminary results, including a new proof of a proposition of H\"urlimann, \cite{H} Prop. 3.3. The group
$\sha(L)$ is defined in Section 5, and the homomorphism $\alpha_c$ in Section 6. In both sections, we start with the case where the \'etale algebra $L$ has
a cyclic factor of prime power degree, which is the essential case. We also show how one can reduce the exponent of the prime number, using the exact
sequence of Proposition \ref{exact} - this is then used in inductive arguments. The main result is proved in Section 7 (see Theorem \ref{main_theo}).
Section 8 contains the application of the above results to the special case where all the factors of the \'etale algebra are cyclic.

\medskip
Note that the results of this paper are related to the Brauer-Manin obstruction. Indeed,
for $c = 1$, the equation $(\ref{e0})$ yields the so-called {\it norm-one-torus} defined by $L/k$  (see 1.2 for details); we denote this torus by $T_{L/k}$.
When $k$ is an algebraic number field, then one can  deduce from \cite{San} that the only obstruction to the Hasse principle is the Brauer-Manin obstruction, and is
an element of the group $\sha^2(k,\hat{T}_{L/k})^*$. We show that $\sha(L) \simeq \sha^2(k,\hat{T}_{L/k})$ (see Proposition \ref {sha_isom}), hence our results
provide an explicit description of the Brauer-Manin obstruction.

\medskip
We thank the referee for several very useful comments. The second named author wishes to acknowledge support from the Alexander von Humboldt
Foundation. The third named author is partially supported by the NSF grants DMS 1801951 and FRG 1463882.

\section{ Notation, definitions and basic facts}

\medskip
\subsection {Weil restriction}
If $f : R \to R'$ is a homomorphism of commutative rings such that $R'$ is a projective $R$-module of finite type, and if $W$ is an affine $R'$-scheme, then we denote by
$\rR_{R'/R} W$ the Weil restriction (see for instance [O 84], Appendice 2).

\subsection{Etale algebras, tori and characters}

Let $k$ be a field, let $k_s$ be a separable closure of $k$ and set $\Gamma_k = {\rm Gal}(k_s/k)$. We fix once and for all this
separable closure $k_s$, and all separable extensions of $k$ that will appear in the paper will be contained in $k_s$.
We use standard notation in Galois cohomology; in particular, if $M$ is
a discrete $\Gamma_k$-module and $i$ is an integer  $\ge 0$, we set $H^i(k,M) = H^i(\Gamma_k,M)$.

\medskip
If $L$ is a commutative \'etale $k$-algebra of finite rank, we denote by $N_{L/k}$  the norm map, and set $T_{L/k} = \rR^{(1)}_{L/k}(\gm_m)$; then
$T_{L/k}$ is the $k$-torus determined by the exact sequence
 \begin{equation}\label{e2}
\xymatrix@C=0.5cm{
  1 \ar[r] &   T_{L/k}   \ar[r] & \rR_{L/k} (\gm_m)   \xrightarrow{N_{L/k}} \gm_m \ar[r] & 1 }.
\end{equation}

\medskip

For a $k$-torus $ T$, we denote by $\hat{ T} = {\rm Hom}(T,\gm_m)$ its character group. If $K/k$ is a finite separable extension,
set $\Gamma_K=\Gal(k_s/K)$. If moreover
$M$ is a discrete $\Gamma_K$-module, set  $\rI_{K/k}(M)=\ind_{\Gamma_K}^{\Gamma_k}(M)$.

\medskip
The following lemmas will be used several times in the sequel

\medskip
\noindent
\begin{lemma}\label{product}
Let $F/k$ be a separable extension of finite degree, and let $L$ be the product of $n$ copies of $F$. Then we have

\smallskip
{\rm (i)} \ \ \ $T_{L/k} \simeq \rR_{F/k}(\gm_m)^{n-1} \times T_{F/k}$.

\smallskip
{\rm (ii)} \ \ \
$H^1(k,\hat{T}_{L/k})\simeq H^1(k,\hat{T}_{F/k}).$
\end{lemma}

\medskip
\noindent
{\bf Proof.} The isomorphism $(\rR_{F/k}(\gm_m))^{n}\to (\rR_{F/k}(\gm_m))^{n}$
sending
$(b_1,...,b_{n})$ to $(b_1,...,b_{n-1},b_1... b_{n-1}b_{n})$ induces
an isomorphism  $T_{L/k} \simeq \rR_{F/k}(\gm_m)^{n-1} \times T_{F/k}$. This proves {\rm (i)}. By  {\rm (i)}, we have $\hat T_{L/k} \simeq \rI_{F/k}( \ent)^{n-1} \oplus \hat T_{F/k}$;
since $H^1(k, \rI_{F/k}( \ent))= 0$,
this implies \rm (ii).

\medskip

\medskip
\noindent
\begin{lemma}\label{cyclic}
Let $K/k$ be a cyclic extension of degree $d$. Then there exists an isomorphism  $$H^1(k,\hat{T}_{K/k})\to  \ent/d {\ent}$$ which is
functorial with respect to base change.
\end{lemma}

\begin{proof} Let $\sigma$ be a generator of ${\rm Gal}(K/k)$.
Consider the exact sequence

$$1 \ra \gm_m \ra \rR_{K/k}(\gm_m) \ra  T_{K/k}  \ra 1,$$

\smallskip \noindent
where the map from  $\rR_{K/k}(\gm_m)$ to $ T_{K/k}$ sends $x$ to $x/\sigma(x)$, and its dual sequence

$$0 \ra \hat{T}_{K/k} \ra \rI_{K/k}(\ent) \ra \ent \ra 0.$$
This exact sequence induces
$$\xymatrix{
  \rI_{K/k}(\ent)^{\Gamma_k}  \xrightarrow{\mathrm{\epsilon}}  \ent \ar[r] & H^1(k,\hat T_{K/k}) \ar[r] & H^1(k,\rI_{K/k}(\ent))=0
  }.$$

We have  $\rI_{K/k}(\ent)^{\Gamma_k}\simeq\ent$, generated by the sum of the elements of ${\rm Gal}(K/k)$, and
the map $\mathrm{\epsilon}$ is  multiplication by $d$; hence we obtain an isomorphism $H^1(k,\hat{T}_{K/k})\to \ent/d\ent$
which is independent of the choice of the generator $\sigma$.

\end{proof}

\medskip

\subsection {The multinorm problem}

\smallskip Let $L$ be an \'etale $k$-algebra, and let $c \in k^*$. Let $X_c$ be the affine $k$-variety determined by  the equation
$N_{L/k}(t) = c$. Then $X_c$ is a torsor under the torus $T_{L/k}$ defined in 1.2, hence defines a class $[X_c] \in H^1(k,T_{L/k})$;
the variety $X_c$ has a $k$-point if and only if $[X_c] = 0$.  Hence we have

$$c \in N_{L/k}(L^{\times}) \iff X_c(k) \not = \emptyset \iff [X_c]= 0.$$

\section{ A construction}

\medskip Let $k$ be a field, and let $L$ be a commutative \'etale $k$-algebra of finite rank; assume that $L$ {\it is not a field}. We keep the notation of the previous
section. The aim of this section is to introduce a $k$-torus that will play a basic role in the study of the
cohomology of the torus $T_{L/k}$, and of the multinorm problem.

\medskip
Let us write $L = K \times K'$, where $K$ and $K'$ are \'etale $k$-algebras, and set $E = K \otimes_k K'$.

\medskip The norm maps $N_{K/k} : K \to k$ and $N_{K'/k} : K' \to k$ induce
$N_{E/K'} : E \to K$ and $N_{E/K} : E \to K'$.
Let
$f:\rR_{E/k}(\gm_m)\ra \rR_{L/k}(\gm_m)$ be defined by $f (x) = (N_{E/K}(x)^{-1},N_{E/K'}(x))$.
It is clear that
the image of $f$ is contained in $T_{L/k}$. Moreover, $f$ is surjective as a map of algebraic groups (easily checked after base change
to the separable closure $k_s$ of $k$).

\medskip
Consider the torus $S_{K,K'}$ defined by the exact sequence

$$\xymatrix@C=0.5cm{
  1 \ar[r] & S_{K,K'} \ar[r] & \rR_{E/k}(\gm_m) \xrightarrow{f} T_{L/k} \ar[r] & 1 }.$$

  Note that $S_{K,K'}$ also fits in the exact sequence

  \begin{equation}\label{e3}
\xymatrix@C=0.5cm{
  1 \ar[r] &  S_{K,K'} \ar[r] & \rR_{K'/k}(T_{E/K'}) \xrightarrow{N_{E/K}} T_{K/k}  \ar[r] & 1 }.
\end{equation}
where $T_{E/K'}$ is defined by the exact sequence of $K'$-tori

$$1 \to T_{E/K'} \to R_{E/K'}(\gm_m) {\buildrel {N_{E/K'}} \over  \longrightarrow} \gm_m \to 1.$$

  \section{Tate-Shafarevich groups}

\medskip We keep the notation of the previous sections, and assume that
$k$  is a global field. Let
$\Omega_k$ be the set of all places of $k$; if $v \in \Omega_k$, we denote by $k_v$ the completion of $k$ at $v$.

\medskip
For any $k$-torus $T$, set $\sha ^i(k,T) = {\rm Ker} (H^i(k,T) \ra \underset {v \in \Omega_k} \prod H^i(k_v,T))$. If $M$ is a $\Gamma_k$-module,
set
$\sha^i (k,M) = {\rm Ker} (H^i(k,M) \ra \underset {v \in \Omega_k} \prod H^i(k_v,M))$. Recall that by  Poitou-Tate duality, we have
$\sha^2(k,\hat T) \simeq \sha^1(k,T)^*$.

\subsection{Hasse principle for the multinorm problem}

\medskip Let $L$ be an \'etale $k$-algebra, and let $c \in k^{\times}$. If $X_c(k_v) \not = \emptyset $ for all $v \in \Omega_k$, then we have $[X_c] \in \sha^1(k,T_{L/k})$.
In particular, the Hasse principle holds for all $c \in k^{\times}$ if and only if $\sha^1(k,T_{L/k}) =0$.

\medskip We have the following relationship between the Tate-Shafarevich groups of the torus $T_{L/k}$, and  the torus $S_{K,K'}$ defined in \S 2 :

\medskip
\noindent
\begin{lemma}\label{T and S}
We have
$\sha^{1}(k,T_{L/k})\simeq\sha^{2}(k,S_{K,K'})$.
\end{lemma}

\medskip
\noindent
{\bf Proof.}
By the definition of the torus $S_{K,K'}$, we have the  exact sequence

$$\xymatrix@C=0.5cm{
  1 \ar[r] & S_{K,K'} \ar[r] & \rR_{E/k}(\gm_m) \xrightarrow{f} T_{L/k} \ar[r] & 1 },$$

\noindent
giving rise to the cohomology exact sequence

$$0\ra H^{1}(k, T_{L/k})\ra H^2(k, S_{K,K'})\ra H^2(k,\rR_{E/k}(\gm_m)).$$

\smallskip
By the corresponding cohomology exact sequence over $k_v$ for all $v \in \Omega_k$ and the Brauer-Hasse-Noether Theorem, we have $\sha^2(k,\rR_{E/k}(\gm_m))=0$, hence
$\sha^{1}(k, T_{L/k})\simeq\sha^{2}(k, S_{K,K'})$, as claimed.

 \medskip

We now compute the group $\sha^2(k, \hat{T}_{K/k})$ for a cyclic extension $K/k$ - note that by Poitou-Tate duality, this is equivalent to
Hasse's cyclic norm principle, which is the following proposition :

\begin{prop}\label{hassecyclicnorm}
Let $K/k$ be a cyclic extension. Then  $\sha^1(k, {T}_{K/k})=0$.
\end{prop}

\begin{proof} We give a proof for the convenience of the reader.
Let $\sigma$ be a generator of ${\rm Gal}(K/k)$.
Consider the exact sequence
$$1 \ra \gm_m \ra \rR_{K/k}(\gm_m) \ra  T_{K/k}  \ra 1,$$

\smallskip \noindent
where the map from  $\rR_{K/k}(\gm_m)$ to $ T_{K/k}$ sends $x$ to $x/\sigma(x)$. This sequence gives rise to an injection $H^1(k,T_{K/k}) \to H^2(k,\gm_m)$.
By the Brauer-Hasse-Noether theorem, we have $\sha^2(k,\gm_m) = 0$, hence $\sha^1(k, {T}_{K/k})=0$.

\end{proof}

\begin{coro}\label{cyclic_sha}
Let $K/k$ be a cyclic extension. Then  $\sha^2(k, \hat{T}_{K/k})=0$.
\end{coro}

This follows from the previous proposition, combined with Poitou-Tate duality.

\section {A result of H\"urlimann}

Using the above lemmas, we  generalize a result of H\"urlimann (\cite{H} Prop. 3.3).

\medskip
\noindent
\begin{prop}\label{hurlimann}
 Let $K/k$ be a cyclic extension of $k$, and let $K'/k$ be a separable extension of finite degree. Let $c \in k^\times$.
 Then the local-global principle holds for the multinorm equation $N_{K/k}(x)N_{K'/k}(y)= c$.
\end{prop}

\medskip
\noindent
{\bf Proof.} Set $L = K \times K'$; the assertion is equivalent to the vanishing of $\sha^1(k,T_{L/k})$.
By Lemma \ref{T and S}, we have  $\sha^1(k,T_{L/k})\simeq\sha^2(k,S_{K,K'}).$ By Poitou-Tate duality we have $\sha^2(k,S_{K,K'})\simeq\sha^1(k,\hat{S}_{K,K'})^\ast,$ hence
it suffices to prove that $\sha^1(k,\hat{ S}_{K,K'})=0$.
Since $K/k$ is a cyclic extension, the algebra $E =K\otimes_k K'$ is isomorphic to a product of copies of $F$, where $F/K'$ is some cyclic field extension. Set $d=[K:k]$ and $f =[F:K']$.

\medskip

Consider the dual sequence  of (\ref{e3}) :
\begin{equation}\label{e6}
\xymatrix{
  0 \ar[r] & \hat{T}_{K/k} \xrightarrow{\iota}  \rI_{K'/k}(\hat{T}_{E/K'}) \xrightarrow{\rho} \hat S_{K,K'} \ar[r] & 0
  },
\end{equation}
and the sequence induced by (\ref{e6})
\begin{equation}\label{e7}
\xymatrix{
   H^1(k,\hat{T}_{K/k}) \xrightarrow{\iota^1} H^1(k,\rI_{K'/k}(\hat{T}_{E/K'})) \xrightarrow{\rho^1} H^1(k,\hat S_{K,K'})  \xrightarrow{\delta}
   H^2(k, \hat{T}_{K/k}).
  }
\end{equation}

We have $H^1(k,\rI_{K'/k}(\hat{T}_{E/K'})) \simeq H^1(K', \hat{T}_{E/K'})$; lemmas  \ref{product}  {\rm (ii)} and  \ref{cyclic} imply that $H^1(K', \hat{T}_{E/K'}) \simeq H^1(K', \hat{T}_{F/K'})\simeq\ent/f\ent.$
The map $\iota^1$ is the natural projection from $\ent/d\ent\to\ent/f\ent$, therefore $\iota^1$ is surjective. This implies that $\delta : H^1(k,\hat S_{K,K'}) \to H^2(k, \hat{T}_{K/k})$
is injective; moreover, $\delta$ induces an injection $\sha^1(k,\hat S_{K,K'}) \to \sha^2(k, \hat{T}_{K/k})$.
Since $K/k$ is a cyclic extension,  we have $\sha^2(k, \hat{T}_{K/k})=0$ by Corollary \ref{cyclic_sha}.
The proposition then follows.

\section{ The group $\sha(K,K')$}

\medskip

We keep the notation of the previous sections : in particular, $L = K \times K'$, where $K$ and $K'$ are \'etale $k$-algebras, and $E = K \otimes K'$.
In addition, we now assume that $K/k$ is a {\it cyclic extension}. Under this hypothesis, we define a finite abelian group $\sha(K,K')$
using the local splitting patterns of $E$, and we show that
$\sha^1(k,T_{L/k})$ is isomorphic to the dual of $\sha(K,K')$.

\medskip Let $K' = \prod_{i \in \cI} K_i$, where the $K_i/k$ are field extensions. Then we have $E = \prod_{i \in \cI} E_i$, with $E_i = K \otimes K_i$.

\medskip
If $v \in \Omega_k$ and if $A$ is a commutative $k$-algebra, set $A^v = A \otimes_k k_v$.

\medskip

\subsection{The  prime power degree case}

Suppose that $K$ is a cyclic extension of degree $p^e$, where $p$ is a prime number.
We start with some notation and definitions.
For each $i \in \cI$, let $M_i$ be a cyclic extension of $K_i$ such that $E_i$ is isomorphic to a product of copies of $M_i$.
Let $p^{e_i}=[M_i:K_i]$; without loss of generality, we assume that $e_i\geq e_{i+1}$ for $1\leq i\leq m-1$.

\medskip

Let $s$ and $t$ be  positive integers. For $s\geq t$, let $\pi_{s,t}$ be the canonical projection  $\ent/p^{s}\ent \to
\ent/p^{t}\ent$.
For $x\in\ent/p^{s}\ent$ and $y\in\ent/p^{t}\ent$, we say that \emph{$x$ dominates $y$} if $s\geq t$ and $\pi_{s,t}(x)=y$; if this is the case, we write
 $x\succeq y$.
For $x\in\ent/p^{s}\ent$ and $y\in\ent/p^{t}\ent$, let  $\delta(x,y)$  be the greatest nonnegative integer $d\leq \min\{s,t\}$
such that $\pi_{s,d}(x)=\pi_{t,d}(y)$.
We have  $\delta(x,y)=\min\{s,t\}$ if and only if $x\succeq y$ or $y\succeq x$.

\medskip

Let $\cI=\{1,...,m\}$. For $a=(a_1,...,a_m)\in\underset{i \in \cI}{\oplus}\ent/p^{e_i}\ent$ and $n\in \ent/p^{e_1}\ent$,
let $I_n(a)$ be the set
$\{i\in\cI|\ n\succeq a_i\}$ and let $I(a)=(I_0(a),...,I_{p^{e_1}-1}(a))$.
\medskip

Let $\cE$ be the set of $p^{e_1}$-tuples $(I_0,...,I_{p^{e_1}-1})$, where $I_0,...,I_{p^{e_1}-1}$ are subsets of $\cI$ such that
$\underset{0\leq n\leq p^{e_1}-1}{\bigcup} I_n=\cI.$
Now we  characterize the image of  the map \[I:\underset{i \in \cI}{\oplus}\ent/p^{e_i}\ent\to \cE.\]
An element $(I_0,...,I_{p^{e_{1}}-1})\in\cE$ is said to be \emph{coherent} if
for all $n_1$, $n_2\in\ent/p^{e_1}\ent$ we have:
\begin{enumerate}
  \item  If $i\in I_{n_1}\cap I_{n_2}$, then $\pi_{e_1,e_i}(n_1)=\pi_{e_1,e_i}(n_2)$.
  \item  If $i\in I_{n_1}$ and $\pi_{e_1,e_i}(n_1)=\pi_{e_1,e_i}(n_2)$, then $i\in I_{n_2}$.
\end{enumerate}

\medskip

Let $\cE^c$ be the subset of all coherent elements in $\cE$.
For $a\in\underset{i \in \cI}{\oplus}\ent/p^{e_i}\ent$, it is clear that $I(a)$ is a coherent element.
Conversely for a coherent element $(I_0,...,I_{p^{e_1}-1})\in\cE^c$, we set $a_i=\pi_{e_1,e_i}(n)$ for $i\in I_n$.
Note that  condition (1) of the definition of a coherent element ensures that  the $a_i$'s are well-defined.
Hence $a=(a_1,...,a_m)$ is a well-defined element in $\underset{i \in \cI}{\oplus}\ent/p^{e_i}\ent$;
condition $(2)$ implies that  $I(a)=(I_0,...,I_{p^{e_1}-1})$.
This shows that $I$ is a bijection between $\underset{i \in \cI}{\oplus}\ent/p^{e_i}\ent$ and $\cE^c$.

\medskip If $w$ is a place of $K_i$, we denote by $K_i^w$ the completion of $K_i$ at $w$.
Given a positive integer $0\leq d\leq e$ and $i \in \cI$, let $\Sigma_i^d$ be the set of all places $v\in\Omega_k$ such that at each place $w$ of $K_i$ above $v$, the algebra
$K \otimes K_i^w$ is isomorphic to a product of isomorphic field extensions of degree at most $p^d$ of $K^w_i$.  Let $\Sigma_i=\Sigma^0_i$, in other words,
$\Sigma_i$ is the set of all places $v \in \Omega_k$ where $E_i^v$ is isomorphic to a product of copies of $K_i^v$.

\medskip
Let $(I_0,...,I_{p^{e_1}-1})\in\cE^c$. For $n_1\in \ent/p^{e_1}\ent$ and $i\in \cI$, set $\delta(n_1,i)=\delta(n_1,\pi_{e_1,e_i}(n_2))$, where $n_2$ is an element in $\ent/p^{e_1}\ent$ such that $i\in I_{n_2}$. Since $(I_0,...,I_{p^{e_1}-1})$ is coherent, $\delta(n_1,i)$ is independent of the choice of $n_2$ and hence is  well-defined.
Note that if we let $a=(a_1,...,a_m)$ be the element in $\underset{i \in \cI}{\oplus}\ent/p^{e_i}\ent$ corresponding to  $(I_0,...,I_{p^{e_1}-1})$,
then $\delta(n_1,i) = \delta(n_1,a_i)$.
For $I_n\subsetneqq \cI$, define
\begin{equation}\label{e24}
\Omega(I_n)=\underset{i\notin I_n}{\cap}{\Sigma}_i^{\delta(n,i)}.
\end{equation}
For $I_n=\cI$, we set $\Omega(I_n)=\Omega_k$.

\medskip
Set \[G = G_k(K,K') =\{(a_1,...,a_m)\in\underset{i \in \cI}{\oplus}\ent/p^{e_i}\ent |\ \underset{n\in\ent/p^{e_1}\ent}{\bigcup} \Omega(I_n(a))=\Omega_k\}.
\]

\medskip

\noindent
\begin{lemma}\label{group}
{\it The set $G$ is a subgroup of $\underset{i \in \cI}{\oplus}\ent/p^{e_i}\ent$.}
\end{lemma}

\medskip
\noindent
{\bf Proof.}
Let $a=(a_1,...,a_m)$ and $b=(b_1,...,b_m)$ be elements of $G$.
By the definition of $G$, for each  $v\in\Omega_k$, there exist some  $n$, $n'\in\ent/p^{e_1}\ent$ such that $v\in \Omega(I_n(a))$ and $v\in\Omega(I_{n'}(b))$.
We claim that $$v\in \Omega(I_{n+n'}(a+b)).$$
This is clear when  $I_{n+n'}(a+b)=\cI$.
Suppose that $I_{n+n'}(a+b)\neq \cI$.
First note that $\delta(n+n',a_i+b_i)\geq \mathrm{min}\{\delta(n,a_i),\delta(n',b_i)\}$
and that $\mathrm{min}\{\delta(n,a_i),\delta(n',b_i)\} \leq e_i$
for all $i\in\cI$.
Pick an arbitrary $i\notin I_{n+n'}(a+b)$.
Without loss of generality, we suppose that $\mathrm{min}\{\delta(n,a_i),\delta(n',b_i)\}=\delta(n,a_i)$.
If $i\notin I_n(a)$,  we have $v\in\Sigma_i^{\delta(n,a_i)}\subseteq \Sigma_i^{\delta(n+n',a_i+b_i)}$;
hence we have $v\in \Omega(I_{n+n'}(a+b)).$

\medskip If $i\in I_n(a)$, then by definition $\delta(n,a_i) = e_i$. We have  $\delta(n,a_i) \le
 \delta(n',b_i)$ by assumption, hence  $\delta(n',b_i) \ge e_i$. But $\delta(n',b_i) \le e_i$, therefore
 we have $\delta(n',b_i) = e_i$, and hence $i \in I_{n'}(b)$. This implies that
 $i\in I_{n+n'}(a+b)$, and this is a contradiction.
This completes the proof of the lemma.

\medskip
Let $D$ be the subgroup of $\underset{i \in \cI}{\oplus}\ent/p^{e_i}\ent$ generated by the diagonal element $(1,...,1)$, and note that
$D$ is contained in $G$.  Set $$\sha_k(K,K') = G/D.$$

\medskip
\begin{example} Assume that $k  = \rat$, and that $L = \rat ( \sqrt a) \times \rat ( \sqrt b) \times \rat ( \sqrt {ab})$, where $a,b$ are distinct square-free integers.
Set $K = \rat ( \sqrt a)$, $K_1 = \rat (\sqrt b)$ and $K_2 = \rat \sqrt {ab})$. Then with the above notation we have $\cI = \{1,2 \}$,
and $E_1 = E_2 = \rat (\sqrt a, \sqrt b)$, hence $e = e_1 = e_2 =1$.  This implies that either $\sha(K,K') = 0$, or $\sha(K,K') \simeq \ent /2 \ent$. Note that

\medskip

\centerline
{there exists $v \in \Omega_k$ such that $E_1^v$ is a field $\iff$  $\Sigma_1 \cup \Sigma_2 \not = \Omega_k$,}

\medskip
\noindent
hence

\medskip
\centerline {$\sha(K,K') = 0 \iff$ there exists $v \in \Omega_k$ such that $E_1^v$ is a field.}

\medskip Set now $a = 13$, $b = 17$ : then there exists no
$v \in \Omega_k$ such that $E_1^v$ is a field, therefore $\sha(K,K') = \ent / 2\ent$.
Note that it is well-known that the multinorm principle fails in this case (see for instance [CT 14], Proposition 5.1).

\end{example}

%\medskip

%\begin{lemma}\label{common subfield}
%If all the factors of $L$ contain a common subfield $F$, then $\sha_F(K,K') = \sha_k(K,K')$.

%\end{lemma}

%\medskip
%\noindent
%{\bf Proof.} Since all the factors of $L$ contain $F$, the tensor product $K\otimes_k K_i$ is a product of copies of $K \otimes_F K_i$.
%Using this, we see that $G_k(K,K_i) = G_F(K,K_i)$, hence $\sha_F(K,K') = \sha_k(K,K')$.

\begin{theo}\label{sha_T0_primepower}
{\it Suppose that $K/k$ is a cyclic extension of degree $p^e$, where $p$ is a prime number. Then $\sha^1(k,\hat S_{K,K'})\simeq \sha(K,K')$.}
\end{theo}
\begin{proof}
Consider the dual sequence  of (\ref{e3}),
\begin{equation}\label{e30}
\xymatrix{
  0 \ar[r] & \hat{T}_{K/k} \xrightarrow{\iota}  \rI_{K'/k}(\hat{T}_{E/K'}) \xrightarrow{\rho} \hat{S}_{K,K'} \ar[r] & 0
  },
\end{equation}
and the exact sequence induced by (\ref{e30}),
\begin{equation}\label{e20}
\xymatrix{
   H^1(k,  \hat{T}_{K/k}) \xrightarrow{\iota^1}   H^1(k, \rI_{K'/k}(\hat{T}_{E/K'})) \xrightarrow{\rho^1} H^1(k,  \hat{S}_{K,K'} ) \ra
   H^2(k, \hat{T}_{K/k})
 . }
\end{equation}

We have $\sha^2(k,  \hat{T}_{K/k}) =0$ by Corollary  \ref{cyclic_sha},
therefore $\sha^1(k,\hat{S}_{K,K'})$ is in the image of $\rho^1$.

\medskip
Note that $H^1(k,\rI_{K_i/k}(\hat{T}_{E_i/K_i}))\simeq H^1(K_i,\hat{T}_{E_i/K_i})$, and that by Lemma \ref {product} {(ii)}, we have
$H^1(K_i,\hat{T}_{E_i/K_i}) \simeq H^1(K_i,\hat{T}_{M_i/K_i})$. Moreover, by Lemma \ref {cyclic}, we have $H^1(K_i,\hat{T}_{M_i/K_i}) \simeq
\ent/p^{e_i}\ent$.

\medskip

In the following we identify $H^1(k,\hat T_{K/k})$ to  $\ent/p^e\ent$ and
$H^1(k,\rI_{K_i/k}(\hat{T}_{E_i/K_i}))$ to $\ent/p^{e_i}\ent$ for $1\leq i\leq m$.
Under this identification, the map

$$\iota^1: H^1(k,\hat T_{K/k}) \ra  H^1(k, \rI_{K'/k}(\hat{T}_{E/K'})) = \underset{i \in \cI}{{\oplus}} H^1(k,\rI_{K_i/k}(\hat{T}_{E_i/K_i}))$$
sends $\ent/p^e\ent$ to $\underset{i \in \cI}{{\oplus}} \ent/p^{e_i}\ent$ by the natural projections.
Therefore we can rewrite the exact sequence $(\ref{e20})$ as follows :
\begin{equation}
\xymatrix{
   \ent/p^{e}\ent \xrightarrow{\iota^1} \underset{i \in \cI}{{\oplus}} \ent/p^{e_i}\ent \xrightarrow{\rho^1} H^1(k,\hat{S}_{K,K'} ) \ra
   H^2(k, \hat{T}_{K/k}),
 }
\end{equation}
where $\iota^1$ is the natural projection from $\ent/p^e\ent$ to $\ent/p^{e_i}\ent$ for each $i$. Note that the image of $\iota^1$ is the subgroup $D$, and we
have the exact sequence

\begin{equation}\label{e21}
\xymatrix{
   0 \ra ( \underset{i \in \cI}{{\oplus}} \ent/p^{e_i}\ent )/D  \xrightarrow{\rho^1} H^1(k,\hat{S}_{K,K'} ) \ra
   H^2(k, \hat{T}_{K/k}).
 }
\end{equation}

\medskip
Let $a=(a_1,...,a_m)\in \underset{i \in \cI}{{\oplus}} \ent/p^{e_i}\ent$ and $[a]$ be its image in $(\underset{i \in \cI}{{\oplus}} \ent/p^{e_i}\ent)/D$. We claim that $\rho^1([a])$ is in $ \sha^1(k,\hat{S}_{K,K'} )$ if and only if $a\in G$.

We denote by $a^v$ the image of $a$ in
$\underset{i=1}{\overset{m�}{\oplus}}H^1(k_v,\rI_{K^v_i/k_v}(\hat{T}_{E^v_i/K^v_i}))$, and by $D_v$ the image of $D$ in this sum.

\medskip
By the exact sequence  (\ref{e21}) over $k_v$, we have $\rho^1([a])\in  \sha^1(k,\hat{S}_{K,K'} )$ if and only if  $a^v \in D_v$ for all places $v \in \Omega_k$. Therefore, it
suffices to prove that $a\in G$ if and only if  $a^v \in D_v$ for all places $v \in \Omega_k$.

\medskip
Suppose that $a\in G$, and let $v \in \Omega_k$. Then there exists $n\in \ent/p^{e_1}\ent$ such that $v\in \Omega(I_n(a))$. If $I_n(a)=\cI$, then clearly $a\in D\subseteq G$.
Suppose that $I_n(a)\neq \cI$. This implies that for
each $i\notin I_n(a)$ and for each place $w$ of $K_i$ above $v$, the \'etale algebra $K_i^w\otimes K$ is isomorphic to a product of field extensions of $K_i^w$ of degree at most $\delta(n,i)$.
Let $\delta_i=\delta(n,i)=\delta(n,a_i)$. Note that \[
H^1(k_v,\rI_{K^v_i/k_v}(\hat{T}_{E^v_i/K^v_i})) = H^1(K^v_i,\hat{T}_{E^v_i/K^v_i}). \]
We have
\[
H^1(K^v_i,\hat{T}_{E^v_i/K^v_i}) \simeq\underset{w|v}{\oplus}H^1(K^w_i,\hat{T}_{K^w_i\otimes K/K^w_i})\simeq
\underset{w|v}{\oplus}\ent/p^{e_{i,w}}\ent,\]
where $e_{i,w}\leq\delta_i$, and the localization map  $H^1(K_i,\hat{T}_{E_i/K_i}) \to H^1(K^v_i,\hat{T}_{E^v_i/K^v_i})$
is the canonical projection $\pi_{e_i,e_{i,w}}$ from $\ent/p^{e_{i}}\ent$ to each component $\ent/p^{e_{i,w}}\ent$.
Since for all $i\notin I_n(a)$ we have $e_{i,w}\leq\delta_i$,  and $\pi_{e_i,\delta_i}(a_i)=\pi_{e_1,\delta_i}(n)$, this implies that  $a^v=(n,...,n)^v$.

\medskip
Suppose conversely that $a^v \in D_v$ for all $v\in\Omega_k$ and $a\notin G$.
Then $a\notin D$, and there exists a place $v\in\Omega_k$ such that  $v \not \in  \underset{n\in\ent/p^{e_1}\ent}{\cup} \Omega(I_n(a))$.
Since $a^v \in D_v$, there exists $n'\in\ent/p^e\ent$ such that $a^v=(\iota^1(n'))_v$. Let $n=\pi_{e,e_1}(n')$.
As $v\notin \Omega (I_{n}(a))$, there exists $i\notin I_{n}(a)$ and a place $w$ of $K_i$ above $v$ such that $K_i^w\otimes K$  is isomorphic
to a product of field extensions of degree $p^{e_{i,w}}$ of $K_i^w$, with $e_{i,w} >\delta_i$.
Then by the definition of $\delta_i=\delta(n,a_i)$,we have $\pi_{e_i,e_{i,w}}(a_i)\neq \pi_{e_1,e_{i,w}}(n)$.
Hence the localization $a_i^v$ of the $i$-th coordinate of $a$ is not equal to the localization  of the $i$-th coordinate of $(n,...,n)$, which is a contradiction.
Our claim then follows.
Therefore, we have $\sha^1(k,\hat{S}_{K,K'})\simeq\sha(K,K')$.

\end{proof}

\begin{coro}\label{coro primepower}
{\it Suppose that $K/k$ is a cyclic extension of degree $p^e$, where $p$ is a prime number. Then $\sha^1(k,T_{K,K'})\simeq \sha(K,K')^*$.}

\end{coro}

\medskip
\noindent
{\bf Proof.}
By Lemma \ref{T and S}, we have $\sha^{1}(k,T_{L/k})\simeq\sha^{2}(k,S_{K,K'})$. Theorem \ref{sha_T0_primepower}  implies that
$\sha^1(k,\hat S_{K,K'})\simeq \sha(K,K')$. By  Poitou-Tate duality,  we have $\sha^{2}(k,S_{K,K'}) \simeq \sha^1(k,\hat S_{K,K'})^*$,
hence the corollary is proved.

\subsection {The group $\sha(K/K_0,K')$}

\medskip
Let $K_0$ be the unique subfield of $K$ such that $[K_0:k] = p^{e-1}$. The proof of the main theorem in the prime power case uses induction on $e$, and the
comparison of the groups $\sha(K,K')$ and $\sha(K_0,K')$.
We first define a homomorphism
$F : \sha(K_0,K') \to \sha(K,K'),$  and then determine the cokernel of $F$, denoted by $\sha(K/K_0,K')$.

\medskip
Note that if $e = 1$, then $K_0 = k$, and hence  $\sha (K_0,K')$ is trivial; in this
case, $\sha(K/K_0,K')$ is the group $\sha(K,K')$ itself.

\bigskip
{\bf The homomorphism $ \sha(K_0,K') \to \sha(K,K').$}

\bigskip

Recall that we have $K' =  \underset{i \in \cI} \prod K_i$, that $E_i = K \otimes K_i$, and that $E_i$ is the product of copies of a cyclic extension of
degree $p^{e_i}$ of $K_i$. Set $E_i^0 = K_0 \otimes K_i$. Then $E_i^0$ also splits as a product of copies of a cyclic extension of $K_i$; let us denote by $p^{f_i}$
the degree of this extension.

\medskip
\begin{prop}\label{e and f}
For all $i \in \cI$, we have $f_i \le e_i$. If moreover $e_i \not = 0$, then $e_i = f_i + 1$.

\end{prop}

This is an immediate consequence of the following proposition :

\begin{prop}\label{E and F} Let $F/k$ be a field extension, and let $K \otimes_k F$ be a product of cyclic field extensions  of $F$ of degree $p^{e_F}$;
let $K_0 \otimes_k F$ be a product of cyclic field extensions of $F$ of degree $p^{f_F}$. Then we have

\medskip
{\rm (i)} $f_F \le e_F$;

\medskip
{\rm (ii)} $f_F \ge e_F - 1$;

\medskip
{\rm (iii)} If $e_F \not = 0$, then $e_F = f_F +1$.

\end{prop}

\noindent
{\bf Proof.} If $n$ is a positive integer, let us denote by $C_n$ the cyclic group of order $n$. Let us consider the homomorphisms $$\Gamma_{F}  {\buildrel \iota \over \rightarrow} \Gamma_k {\buildrel {\phi_K} \over \longrightarrow} C_{p^e} {\buildrel \pi \over \longrightarrow} C_{p^{e-1}} \to 1,$$ where $\iota$ is the inclusion of $\Gamma_{F}$ into $\Gamma_k$, the homomorphism $\phi_K : \Gamma_k \to C_{p^e}$ corresponds to the
cyclic extension $K/k$, and $\pi : C_{p^e} \to C_{p^{e-1}}$ is the quotient of $C_{p^e}$ by its unique subgroup of order $p$. Note that the image
of $\phi_K \circ \iota$ is the Galois group of the cyclic factors of $K \otimes_k F$, and hence is of order $p^{e_F}$; similarly, the image of  $\pi \circ \phi_K \circ \iota$ is
the Galois group of the cyclic factors of $K_0 \otimes_kF$, and hence is of order $p^{f_F}$. Therefore we have $f_F \le e_F$. Moreover, if $e_F \not = 0$, then
the image of $\phi_K \circ \iota$ contains the unique subgroup of order $p$ of $C_{p^e}$, and hence $e_F= f_F + 1$. This
completes the proof of the proposition.

\bigskip
For all $i \in \cI$, let $F_i : \ent/p^{f_i} \ent \to \ent/ p^{e_i} \ent$ be the  inclusion of the subgroup
of order $p^{f_i}$ in the group $ \ent/ p^{e_i} \ent$, and set $F_{K/K_0} = F = \underset{i \in \cI} \oplus F_i$.

\medskip
\begin{prop}\label{sha 0 to sha}
The map $F : \underset {i \in \cI} \oplus \ent/p^{f_i} \ent \to  \underset {i \in \cI} \oplus \ent/ p^{e_i} \ent$  induces an injective
homomorphism $F : \sha(K_0,K') \to \sha (K,K')$.

\end{prop}

\medskip
\noindent
{\bf Proof.} Let us recall some notation from 5.1, for $K$ and $K_0$ :
For all $i \in \cI$ and for all positive integers $d$, we denote by $\Sigma(K)_i^d$ (respectively  $\Sigma(K_0)_i^d$) the set of all places $v\in\Omega_k$ such that at each place $w$ of $K_i$ above $v$, the algebra
$K \otimes K_i^w$  (respectively $K_0 \otimes K_i^w$ ) is isomorphic to a product copies of a cyclic extension of degree at most $p^d$ of $K^w_i$.
Recall that

\[G = G(K,K') =\{ a \in \underset {i \in \cI} \oplus \ent/p^{e_i}\ent  \ \ |    \ \underset{n\in\ent/p^{e_1}\ent}{\bigcup} \Omega (I_n(a))=\Omega_k\},
\]
and that $D$ is the diagonal subgroup of $G$. Similarly, set

\[G_0 = G(K_0,K') =\{ b \in \underset {i \in \cI} \oplus \ent/p^{f_i}\ent  \ \ |    \ \underset{n\in\ent/p^{f_1}\ent}{\bigcup} \Omega (I_n(b))=\Omega_k\},
\]
and let $D_0$, be the diagonal subgroup of $G_0$.  Then we have $\sha(K,K') = G/D$ and $\sha(K_0,K') = G_0/D_0$.

\medskip
Let $b \in G_0$, and let us show that $F(b) \in G$. Let $v \in \Omega_k$. Then there exists $r \in \ent/p^{f_1}\ent$ such that $v \in \Sigma(K_0)_i^{\delta(r,i)}$ for all $i \in \cI$
such that $i \not \in I_r(b)$. Note that for all positive integers $\delta$, we have $\Sigma(K_0)_i^{\delta} \subset \Sigma(K)_i^{\delta + 1}$. Set $n = F_1(r) \in \ent /p^{e_1} \ent$;
then we have $\delta(n,F_i(b_i) ) = \delta (r,b_i) + 1$. Hence we have $v \in \Sigma(K)_i^{\delta (r,b_i) +1}$, and therefore $F(b) \in G$.

\medskip
It is clear that $F$ is injective.

\bigskip

\begin {remark} \label {injective}
For any subextension $N/k$ of $K/k$, let $F_{K/N} : \sha(N,K') \to \sha(K,K')$ be the injective homomorphism obtained by successive applications of Proposition \ref{sha 0 to sha}.
\end {remark}

\bigskip
{\bf The group $\sha(K/K_0,K')$.}

\bigskip
As we will see, the cokernel of $F$ is
isomorphic to the group $\sha(K/K_0,K')$, defined as follows :

\medskip
For all $i \in \cI$, set $r_i = {\rm min} \{1,e_i \}$. For all  $c  \in \underset {i \in \cI} \oplus \ent/p^{r_i}\ent$ and $n \in \ent / p \ent$, set $I^1_n(c) = \{ i \in \cI \ \ | n \succeq  c_i \}$.
If $I^1_n(c) \not = \cI$, set $\Omega (I^1_n(c)) = \underset{i\notin I^1_n(c)}{\cap}{\Sigma}_i$; if $I^1_n(c) = \cI$, set $\Omega (I^1_n(c)) = \Omega_k$. Set

\[G(K/K_0,K') =\{c \in \underset {i \in \cI} \oplus \ent/p^{r_i}\ent  \ \ |    \ \underset{n\in\ent/p^{r_1}\ent}{\bigcup} \Omega (I^1_n(c))=\Omega_k\},
\]
let $D(K/K_0,K')$ be the diagonal subgroup of $G(K/K_0,K')$, and set $$\sha(K/K_0,K') = G(K/K_0,K') / D(K/K_0,K').$$

\medskip
\begin{lemma}\label{K/K_0}
The projection $\pi : \underset {i \in \cI} \oplus \ent/p^{e_i} \ent \to  \underset {i \in \cI} \oplus \ent/ p^{r_i} \ent$  induces a
homomorphism $\pi : \sha(K,K') \to \sha (K/K_0,K')$.

\end{lemma}

\medskip
\noindent
{\bf Proof.} Let $a \in G$, and set $\overline a = \pi(a)$. Let us show that $\overline a \in G(K/K_0,K')$. Let $v \in \Omega_k$; then there exists $s \in \ent / p^{e_1} \ent$ such
that $v \in \Omega(I_s(a))$. Set $n = \pi_{e_1,1}(s)$, and let us prove that $v \in \Omega (I^1_n(\overline a))$. This is clear if $ I^1_n(\overline a) = \cI$. Suppose that $ I^1_n(\overline a)
\not =  \cI$. If $i \in \cI$ is such that $i \not \in  I^1_n(\overline a)$, then we have $i \not \in I_s(a)$, and therefore $v \in \Sigma_i^{\delta(s,a_i)}$. Since $n = \pi_{e_1,1}(s)$ and
$i \not \in  I^1_n(\overline a)$, we have ${\delta(s,a_i)} = 0$, and hence $v \in \Sigma_i$. Therefore we have $\overline a \in G(K/K_0,K')$, as claimed, and this completes
the proof of the lemma.

\medskip
\begin{prop}\label{exact}
The sequence  $$0 \to  \sha(K_0,K') \buildrel {F} \over \longrightarrow \sha (K,K') \buildrel {\pi} \over \longrightarrow \sha(K/K_0,K') \to 0$$
is exact.

\end{prop}

\medskip
\noindent
{\bf Proof.} It is clear that $F$ is injective, and that $\pi \circ F = 0$; it remains to check that $\pi$ is surjective, and that ${\rm Ker}(\pi) \subset {\rm Im}(F)$. Let
us check the second assertion first. Let $a \in \sha(K,K')$ be such that $\pi (a) = 0$. Then there exists $b \in  \underset {i \in \cI} \oplus \ent/p^{f_i} \ent$
such that $F(b) = a$; let us check that $b \in \sha(K_0,K')$. Let $v \in \Omega_k$. Then there exists $n \in \ent/p^{e_1}\ent$ such that $v \in \Omega(I_n(a))$.
If $i \in I_n(a)$, then we have $\pi_{e_1,e_i}(n) = a_i$. Since $a_i = F_i(b_i)$, this implies that there exists $r \in \ent/p^{f_1}\ent$ such that
$n = F_1(r)$ and $I_n(a)=I_r(b)$. Let us show that $v \in \Omega(I_r(b))$. For all $i \in \cI$ such that $i \not \in I_n(a)$, we have $v \in \Sigma(K)_i^{\delta(n,a_i)}$.
Note that $\delta(n,a_i) = \delta(r,b_i) + 1$ and $[K:K_0]=p$. Hence $v \in \Sigma(K)_i^{\delta(n,a_i)}$
implies that $v \in \Sigma(K_0)_i^{\delta(r,b_i)}$. Therefore we have $v \in \Omega(I_r(b))$,
as claimed, and this implies that $b \in \sha(K_0,K')$. Let us now prove that $\pi$ is surjective. Let $\overline a \in \sha(K/K_0,K')$. For each $n \in \ent/p\ent$,
let us fix a lifting $r(n) \in \ent/p^{e_1}\ent$. If $i \in I^1_n(\overline a)$, set $a_i = \pi_{e_1,e_i}(r(n))$. Let us check that $a_i \in \ent/p^{e_i}\ent$ is
well-defined. Suppose that $n_1, n_2 \in \ent/p \ent$ are such that $i \in I^1_{n_1}(\overline a) \cap I^1_{n_2}(\overline a)$; then we have
$\pi_{1,r_i}(n_1) = \pi_{1,r_i}(n_2)$. If $n_1 \not = n_2$, then this implies that $r_i = 0$, hence $e_i = 0$. We have $\pi_{e_1,e_i}(r(n_1)) =
\pi_{e_1,e_i}(r(n_2))$ in this case, hence $a_i$ is well-defined. Let us check that $a \in \sha(K,K')$. Since $\overline a \in \sha(K/K_0,K')$, we have
 $\underset{n\in\ent/p^{r_1}\ent}{\bigcup} \Omega (I^1_n(\overline a))=\Omega_k$. Let $v \in \Omega_k$; then there exists $n \in \ent/p \ent$ such
 that $v \in  \Omega (I^1_n(\overline a))$. Let $r = r(n)$; we claim that $v \in  \Omega (I_r( a))$. If $I^1_n(\overline a) = \cI$, then we have
$I_r(a) = \cI$, and the claim is clear. Suppose that $I^1_n(\overline a) \not = \cI$. If $i \not \in I_r(a)$, then we have $i \not \in I^1_n(\overline a)$ by
construction, hence $v \in \Sigma_i$. Since $\Sigma_i \subset \Sigma_i(K)^{\delta(r,a_i)}$, the claim follows. This completes the proof of
the proposition.

\bigskip
{\bf The group $\sha(K/K_0,K')$ and partitions.}

\medskip
In this section we give some properties of $G(K/K_0,K')$ in terms of partitions of $\cI$.
This will be useful in the proof of the main theorem (Thm. 7.1).

\medskip
\noindent
\begin{lemma}\label{partition} The set $G(K/K_0,K')$ is in bijective correspondence with the partitions $(J_0,...,J_{p-1})$ of
the set $ \{ i \in \cI \ | \ r_i = 1 \}$ such that $\underset{n\in\ent/p\ent}{\cup} \Omega(J_n)=\Omega_k$.

\end{lemma}

\medskip
\noindent
{\bf Proof.}
Recall that we have $r_i = 0$ or $1$. Set $\cI' = \{ i \in \cI \ | \ r_i = 1 \}$. Let
$a \in G(K/K_0,K')$, and set $I^{1}({a})\cap \cI' =(I_{0}^1(a)\cap \cI',...,I^{1}_{p-1}(a)\cap \cI')$; note
 that $I^{1}({a})\cap \cI'$ is a partition of $\cI'$.
Hence the set $G(K/K_0,K')$ is then in bijective correspondence with the partitions $(J_0,...,J_{p-1})$ of $\cI'$ such that $\underset{n\in\ent/p\ent}{\cup} \Omega(J_n)=\Omega_k$, as claimed.

\medskip
In the sequel, we identify $G(K/K_0,K')$ with the set of these partitions. We also note a consequence for the case where $K$ is of degree $p$, in other words,
if $e = 1$. If $K/k$
is of prime degree, then either $E_i$ is a field extension of $K_i$, or $E_i$ is a product of copies of $K_i$. Let $J$ be the subset of $I$ such that $E_i$ is a field extension
of $K_i$ if $i \in J$, and that $E_i$ is a product of copies of $K_i$ if $i \not \in J$.

\medskip
\noindent
\begin{lemma}\label{primedegreepartition} Assume that $K/k$ is a degree $p$ extension. Then $G(K,K')$ is in bijective correspondence with
the partitions $(J_0,...,J_{p-1})$ of $J$
such that $\underset{n\in\ent/p\ent}{\cup} \Omega(J_n)=\Omega_k$.

\end{lemma}

\medskip
\noindent
{\bf Proof.} Since $e=1$, we have $K_0 = k$ and $G(K/K_0,K') = G(K,K')$. Hence the lemma follows from Lemma \ref{partition}.

\bigskip

Let $K_{\rm prim}$  be the unique subfield of $K$ of degree $p$ over $k$.

\medskip
\begin{prop}\label{prim} Assume that $K_{\rm prim}$ is linearly disjoint from $K_i$ for all $i\in I$. Then the group $\sha(K/K_0,K')$ is a subgroup of $\sha(K_{\rm prim},K')$.

\end{prop}

\medskip
\noindent
{\bf Proof.}
%$J$ be the subset of $i \in \cI$ such that $K_{\rm prim} \otimes_k K_i$ is a field extension.
Note that $G(K_{\rm prim},K')$ is in bijection with the set of partitions $(I_0,\dots,I_{p-1})$ of $I$
such that $\underset{n\in\ent/p\ent}{\cup} \Omega_{K_{\rm prim}}(I_n)=\Omega_k$, where
$ \Omega_{K_{\rm prim}}(I_n) = \cap_{i \not \in I_n} \Sigma_i(K_{\rm prim})$ (see Lemma
\ref{primedegreepartition}).

\medskip
As $K_{\rm prime}$ is linearly disjoint from $K_i$, $r_i=1$ for all $i\in I$. Hence by Lemma
\ref{partition}, the set
$G(K/K_0,K')$ is in bijection with the set of partitions $(I_0,\dots,I_{p-1})$ of $I$
such that $\underset{n\in\ent/p\ent}{\cup} \Omega_{K}(I_n)=\Omega_k$, where
$ \Omega_{K}(I_n) = \cap_{i \not \in I_n} \Sigma_i(K)$.

\medskip Note that $\Sigma_i(K_{\rm prim}) \supseteq \Sigma_i(K)$ for all $i \in \cI$. Hence
$G(K/K_0,K') \subset G(K_{\rm prim},K')$, and this implies that
$\sha(K/K_0,K')$ is a subgroup of $\sha(K_{\rm prim},K')$.

\medskip
\begin{prop}\label{prim prim} Assume that $K_{\rm prim}$ is linearly disjoint from $K_i$ for all $i\in I$. If
 $\sha(K_{\rm prim},K') = 0$,  then $\sha(K,K') = 0$.

\end{prop}

\medskip
\noindent
{\bf Proof.} By Proposition \ref{prim}, we have $\sha(K/K_0,K') =0$; hence
Proposition \ref{exact} implies that $\sha(K,K') = \sha(K_0,K')$. Repeating this argument,
 we see that $\sha(K,K') = \sha(K_{\rm prim},K')$. But $\sha(K_{\rm prim},K') = 0$ by
 hypothesis, hence $\sha(K,K') = 0$, as claimed.

\medskip The following lemma will be useful in the sequel.
\noindent
\begin{lemma}\label{simpl}
{\it Let $(I_0,\dots,I_{p-1}) \in G(K/K_0,K')$, and let $r, r'$ be two distinct elements of $\ent/p\ent$. Set $J_r = I_r$, $J_{r'} = \underset {n \not = r} \cup I_n$, and
$J_n = \emptyset$ if $n \not = r,r'$. Then $(J_0,\dots,J_{p-1}) \in G(K/K_0,K')$. If moreover  $(I_0,\dots,I_{p-1}) \not  \in D(K/K_0,K')$ and $I_r \not = \emptyset$, then
$(J_0,\dots,J_{p-1}) \not  \in D(K/K_0,K')$.}
\end{lemma}

\noindent
{\bf Proof.} Let us show that $\Omega(J_r) \cup \Omega(J_{r'}) = \Omega_k$. Let $v \in \Omega_k$ be such that $v \not \in \Omega(J_r)$. Since we
have $\underset{n\in\ent/p\ent}{\cup} \Omega(I_n)=\Omega_k$, there exists $n(v) \in \ent/p \ent$ with $n(v) \not = r$ such that $v \in \Omega (I_{n(v)})$.
Since $n(v) \not = r$, we have $\Omega (I_{n(v)}) \subset \underset {i \in I_r} \cap \Sigma_i = \Omega (J_{n'})$. Therefore we have $\Omega(J_r) \cup \Omega(J_{r'}) = \Omega_k$, and hence  $(J_0,\dots,J_{p-1}) \in G(K/K_0,K')$.

\medskip
Let us prove the second statement. If $(J_0,\dots,J_{p-1})  \in D(K/K_0)$, then either $J_r = \cI'$ or $J_{r'} = \cI'$; we have $I_r = \cI'$ in the first case,
hence $(I_0,\dots,I_{p-1})   \in D(K/K_0,K')$, and $I_r = \emptyset$ in the second case. This completes the proof of the lemma.

\subsection{The general case}

Recall that $K/k$ is a cyclic extension of degree $d$, and let $\cP$ be the set of prime numbers dividing $d$. For all $p \in \cP$,
 let $K(p)$ be the largest subfield of $K$ such that $[K(p):k]$ is a power of $p$ and set $d(p)=[K(p):k]$. Set
$$\sha (K,K') =  \underset{p \in \cP}{\oplus}  \sha(K(p),K').$$

\begin{prop}\label{sha_isom} We have $\sha^1(k,\hat S_{K,K'})\simeq \sha(K,K')$.

\end{prop}

\medskip
\noindent
{\bf Proof.} By \ref {sha_T0_primepower} we have $\sha^1(k,\hat S_{K(p),K'})\simeq \sha(K(p),K')$, hence it suffices to show that
$\sha^1(k,\hat S_{K,K'})\simeq \underset{p \in \cP}{\prod}  \ \sha^1(k,\hat S_{K(p),K'}).$ For every $p \in \cP$, set $E(p) = K(p) \otimes_k K'$ and $L(p) = K(p) \times K'$.
The inclusion $K(p) \to K$ induces maps $\epsilon_p : T_{K(p)/k} { \to} T_{K/k}$, $\epsilon_p : T_{E(p)/K'} \to T_{E/K'}$ and $\epsilon_p :  S_{K(p),K'} \to S_{K,K'}$. We have the commutative diagram, coming from cohomology exact sequences
associated to the dual sequences of $(\ref{e3})$ :

\begin{equation}\label{e32}
\xymatrix{
   H^1(k,\hat{T}_{K/k}) \ar[d]_{\oplus \hat\epsilon_p^1} \ar[r]^-{\iota^1} & H^1(k,\rI_{K'/k}(\hat{T}_{E/K'})) \ar[d]_{\oplus\hat\epsilon_p^1} \ar[r]^-{\rho^1}& H^1(k,\hat{S}_{K,K'}) \ar[d]_{\oplus\hat\epsilon_p^1} \ar[r]&
   ...\\
   \underset{p \in \cP}{\oplus}H^1(k,\hat{T}_{K(p)/k}) \ar[r]^-{\iota^1} & \underset{p \in \cP}{\oplus}H^1(k,\rI_{K'/k}(\hat{T}_{E(p)/K'})) \ar[r]^-{\rho^1}&
   \underset{p \in \cP}{\oplus} H^1(k,\hat{S}_{K(p),K'}) \ar[r]&
    ...}
  \end{equation} where the vertical maps are induced by the maps $\epsilon_p$. Set $\hat\epsilon^1=\oplus\hat\epsilon_p^1$.

\medskip For all $i \in \cI$, let $M_i$ be a cyclic extension of $K_i$ such that $E_i$ is isomorphic to a product of copies of $M_i$, and let $d_i = [M_i:K_i]$. If $p$ is a prime
divisor of $[K:k]$, set $E_i(p) = K(p) \otimes_k K_i$, and let $M_i(p)$ be a cyclic extension of $K_i$ such that $E_i(p)$ is isomorphic to a product of copies of $M_i(p)$; set
$d_i(p) = [M_i(p):K_i]$. Note that $d_i(p)$ is the highest power of $p$ dividing $d_i$, and that $d_i = \underset {p \in \cP} \prod d_i(p)$.

\medskip
Note that $H^1(k,\rI_{K_i/k}(\hat{T}_{E_i/K_i}))\simeq H^1(K_i,\hat{T}_{E_i/K_i})$, and that by Lemma \ref {product} {(ii)}, we have
$H^1(K_i,\hat{T}_{E_i/K_i}) \simeq H^1(K_i,\hat{T}_{M_i/K_i})$. Moreover, by Lemma \ref {cyclic}, we have $H^1(K_i,\hat{T}_{M_i/K_i}) \simeq
\ent/{d_i}\ent$. Similarly, we have $H^1(k,\rI_{K_i(p)/k}(\hat{T}_{E_i(p)/K_i}))\simeq \ent/{d_i(p)}\ent$.
Note that the morphism $\epsilon^1_p$ restricted to each $H^1(k,\rI_{K_i/k}(\hat{T}_{E_{i}/K_i}))$
\[\hat\epsilon^1_p :
H^1(k,\rI_{K_i/k}(\hat{T}_{E_{i}/K_i}))\simeq \ent/d_i\ent\to
H^1(k,\rI_{K_i/k}(\hat{T}_{E_{i}(p)/K_i}))\simeq \ent/d_{i}(p)\ent
\]
is the canonical projection $\ent/d_i\ent\to\ent/d_{i}(p)\ent$.
%\[\hat\epsilon^1(p) :
%\underset{i \in \cI}{{\oplus}}H^1(k,\rI_{K_i/k}(\hat{T}_{E_{i}/L_i}))\simeq\underset{i \in \cI}{{\oplus}}\ent/d_i\ent\to
%\underset{i \in \cI}{{\oplus}}H^1(k,\rI_{K_i/k}(\hat{T}_{E_{i}(p)/K_i}))\simeq
%\underset{i \in \cI}{{\oplus}}\ent/d_{i}(p)\ent
%\]
%is the canonical projection $\ent/d_i\ent\to\ent/d_{i}(p)\ent$ on each component.
%Hence the morphism
%\[\underset{p \in \cP}{\oplus}\hat\epsilon^1(p) :\underset{i \in \cI}{{\oplus}}H^1(k,\rI_{L_i/k}(\hat{\rS}_{E_{i}/L_i}))\to
%\underset{j=1}{\overset{s}{\oplus}}\underset{i \in \cI}{{\oplus}}H^1(k,\rI_{L_i/k}(\hat{\rS}_{E_{j,i}/L_i}))\] is an isomorphism.
%Similarly, \[\underset{j=1}{\overset{s}{\oplus}}\hat\epsilon^1_j:H^1(k,\hat{\rS}_{K/k})\simeq\ent/q\ent\to
%\underset{j=1}{\overset{s}{\oplus}}H^1(k,\hat{\rS}_{L_{0,j}/k})\simeq \underset{j=1}{\overset{s}{\oplus}}\ent/q_i\ent\] is also an isomorphism.
Hence by the Chinese reminder theorem, the morphism $\hat\epsilon^1$ restricted to each $H^1(k,\rI_{K_i/k}(\hat{T}_{E_{i}/K_i}))$
\[\underset{p \in \cP}{\oplus} \ \hat\epsilon^1_p : H^1(k,\rI_{K_i/k}(\hat{T}_{E_{i}/K_i}))\to
\underset{p \in \cP}{\oplus}H^1(k,\rI_{K_i/k}(\hat{T}_{E_{i}(p)/K_i}))\] is an isomorphism.
Similarly, \[\underset{p \in \cP}{\oplus} \ \hat\epsilon^1_p : H^1(k,\hat{T}_{K/k})\simeq\ent/d\ent\to
\underset{p \in \cP}{\oplus}H^1(k,\hat{T}_{K(p)/k})\simeq \underset{p \in \cP} {\oplus}\ent/d(p)\ent\] is also an isomorphism.

\medskip

By Corollary  \ref{cyclic_sha}, we have
$\sha^2(k, \hat{T}_{K/k})=0$ and $\sha^2(k, \hat{T}_{K(p)/k} ) =0$, hence $\sha^1(k,\hat{S}_{K,K'})$  and $\sha^1(k,\hat{S}_{K(p),K'})$ are in the image of  the maps $\rho^1$.
Now we show that $\hat\epsilon^1:\sha^1(k,\hat{S}_{K,K'})\to\underset{p \in \cP}{\oplus}\sha^1(k,\hat{S}_{K(p),K'})$
is injective.
Let $\beta\in\sha^1(k,\hat{S}_{K,K'})$. Suppose that $\hat\epsilon^1(\beta)=0$.
As $\sha^1(k,\hat{S}_{K,K'})$ is in the image of $\rho^1$, there is $\gamma\in H^1(k,\rI_{K'/k}(\hat{T}_{E/K'}))$ such that $\rho^1(\gamma)=\beta$.
By the commutativity of the diagram, we have $\rho^1(\hat\epsilon^1(\gamma))=\hat\epsilon^1(\rho^1(\gamma))=0$.
Hence there is $\zeta\in\underset{p \in \cP}{\oplus}H^1(k,\hat{T}_{K(p)/k})$ such that $\iota^1(\zeta)=\hat\epsilon^1(\gamma)$. Then $\hat\epsilon^1\circ\iota^1\circ(\hat\epsilon^1)^{-1}(\zeta)=\hat\epsilon^1(\gamma)$.
As \[\hat\epsilon^1 : H^1(k,\rI_{K'/k}(\hat{T}_{E/K'}))\to
\underset{p \in \cP}{\oplus}H^1(k,\rI_{K'/k}(\hat{T}_{E(p)/K'}))\] is an isomorphism,
we have $\iota^1\circ(\hat\epsilon^1)^{-1}(\zeta)=\gamma$ and hence $\beta=\rho^1(\gamma)=0$.
This proves the injectivity.

\medskip

%%Let $\beta_p\in\sha^1(k,\hat{S}_{K(p),K'})$ for $p\in\cP$.
As
$\sha^2(k, \hat{T}_{K(p)/k} ) =0$ by Corollary \ref{cyclic_sha}, the group $\sha^1(k,\hat{S}_{K(p),K'})$
is in the image of  the maps $\rho^1$ for all $p\in \cP$.
Since $H^1(k,\rI_{K_i/k}(\hat{T}_{E_{i}/K_i}))\to
\underset{p \in \cP}{\oplus}H^1(k,\rI_{K_i/k}(\hat{T}_{E_{i}(p)/K_i}))$ is an isomorphism,
 we see that  $$\hat\epsilon^1:\sha^1(k,\hat{S}_{K,K'})\to\underset{p \in \cP}{\oplus}\sha^1(k,\hat{S}_{K(p),K'})$$
 is surjective. This completes the proof of the proposition.

\bigskip
Note that the proposition, together with Lemma \ref{T and S}, implies that $\sha(K,K')$ does not depend on the decomposition of $L$ as $L = K \times K'$. We will also use
the notation $\sha(L) = \sha(K,K')$, where $L = K \times K'$ is any decomposition of $L$ with $K/k$ a cyclic extension.

\medskip
In summary, we proved

\medskip
\begin{coro}\label{5.10} We have $\sha(L)^* \simeq \sha^1(k,T_{L/k})$.

\end{coro}

\medskip
\begin {example} Let $p$ and $q$ be two distinct odd prime numbers, with $p > q$. For all positive integers $n$, let $\zeta_n$ be a primitive $n$th root of unity. Let $k = {\bf Q}$, and
 $$L = \rat (\zeta_{p^2} ) \times \rat (\zeta_{pq}) \times \rat (\zeta_{q^2} ).$$
 Since $\rat (\zeta_{p^2}) $ and $\rat (\zeta_{q^2} )$ are both cyclic, we can determine $\sha(L)$ in two ways; this shows that the order of $\sha(L)$ divides
 $p-1$, and that
 %Computing $\sha(L)$ by setting
 %$K = \rat (\zeta_{p^2} )$, $K' = \rat (\zeta_{pq}) \times \rat (\zeta_{q^2} )$ and
  $$\sha(L) = \sha(\rat (\zeta_{p} ) \times \rat (\zeta_{pq}) \times \rat (\zeta_{q^2} )).$$
 But since $ \rat (\zeta_{p}) $ is a subfield of $\rat (\zeta_{pq})$, we have $ \sha(\rat (\zeta_{p} ) \times \rat (\zeta_{pq}) \times \rat (\zeta_{q^2} ))
 =  \sha(\rat (\zeta_{p} ) \times \rat (\zeta_{q^2} ))$. Note that by Proposition \ref{hurlimann} we have $\sha(\rat (\zeta_{p} ) \times \rat (\zeta_{q^2} )) = 0$, hence we have
 $$\sha(L) = 0.$$

\end {example}

\section{The Brauer-Manin map}

\medskip

We keep the notation of the previous section : in particular,  $L = K \times K'$, where $K$ is a cyclic extension of $k$ of degree $d$, $K'$ is an \'etale $k$-algebra, and $E = K \otimes K'$.
We write $K' = \prod_{i \in \cI} K_i$, where the $K_i/k$ are field extensions, and $E = \prod_{i \in \cI} E_i$, with $E_i = K \otimes K_i$.
The group $\sha(L) = \sha(K,K')$ is defined in the previous section.

\medskip

Let $c \in k^{\times}$ and recall that $X_c$ is the affine $k$-variety defined by the equation $$N_{L/k}(t)=c.$$

Assume that  $X_c(k_v) \not  = \emptyset$ for all $v \in \Omega_k$. In the following, we define a homomorphism $\alpha_c : \sha(L) \to \rat/\ent$ such that
$X_c(k) \not = \emptyset$ if and only if $\alpha_c = 0$;  the map $\alpha_c$ will be called the {\it Brauer-Manin map} associated to $c$.
We choose this terminology, because this map is an analog of the map given by the Brauer-Manin pairing on $X_c$.

\medskip
\subsection{Local points}

\medskip We start with some preliminary results. We are assuming that ${\prod}   X_c(k_v) \not = \emptyset$; as we will see, this set contains elements
satisfying certain finiteness
conditions. More precisely, we introduce the notion of {\it local points} - these can be thought of as adelic points of $X_c$.

\medskip We first recall the notion of cyclic algebra.
Let us choose a generator $g$ of the cyclic group ${\rm Gal}(K/k)$, and let $\phi : \Gamma_k \to \ent/d\ent$ be given by the composition of the isomorphism ${\rm Gal}(K/k) \to \ent/d\ent$ sending $g$ to $1$ with the surjection $\Gamma_k \to {\rm Gal}(K/k)$. Let us consider the exact sequence $0 \to \ent  {\buildrel {\times d} \over \longrightarrow} \ent \to \ent/d\ent
\to 0$,
and let $\delta : H^1(k,\ent / d\ent) \to H^2(k, \ent)$ be the connecting homomorphism of the associated cohomology exact sequence. If $c \in k^{\times}$, let us denote by $(c)$ the
corresponding element of $H^0(k,\gm_m$). The cup product  $\delta(\phi).(c)$ is an element of $H^2(k,\gm_m)$, and via the identification $H^2(k,\gm_m) \simeq {\rm Br}(k)$ it is
mapped to the class of the cyclic algebra defined by $K$ and $c$ (see for instance \cite {GS}, Proposition 4.7.3). We denote this cyclic algebra by $(K,c)$.

\medskip

The first observation is the following:

\begin{lemma}\label{split-inv}
{\it
Suppose that $E_i$ is isomorphic to a product of copies of a field $M_i$, and set $[M_i:K_i]=d_i$. Then  for any $x \in K_i^\times$, the order of the cyclic algebra $(K,N_{K_i/k}(x))$  divides $d_i$. In particular, if $d_i=1$, then  for any $x \in K_i^\times$, the  algebra $(K,N_{K_i/k}(x))$ splits. }
\end{lemma}

\begin{proof}
Given $x \in K_i^\times$, consider the class of the cyclic algebra $(M_i,x) =\delta(\phi|_{K_i}).(x)$, where $\phi|_{K_i} : \Gamma_{K_i} \to \ent / d\ent$ is the restriction of $\phi$ to $\Gamma_{K_i}$.
Let $r$ be the order of $(M_i,x)$ in $\Br(K_i)$.
Since $M_i$ is of degree $d_i$ over $K_i$,  we have $r|d_i$. By the projection formula (\cite{GS} Prop. 3.4.10), the corestriction of $(M_i,x)$ is $(K,N_{K_i/k}(x))$. Therefore
the order of $(K,N_{K_i/k}(x))$ divides $d_i$.

\end{proof}

\medskip

Let $x = (x^v) \in\underset{v\in\Omega_k}{\prod}   X_c(k_v)$,  and let us write $x^v = (x_0^v,x_1^v,\dots,x_m^v)$, with $x_0^v \in K^v$ and $x_i^v \in K_i^v$ for $i \in \cI$.
%For all $i \in \cI$, set  $y_i^v=N_{K_i^v/k_v}(x_i^v)$.
Let us consider the invariant map  $\inv:\Br(k_v)\to\rat/\ent$
and  set $b^v_i(x) = b^v_i(x^v_i) = \inv(K^v,N_{K_i^v/k_v}(x_i^v))\in \frac{1}{d_i}\ent/\ent\subset\rat/\ent$. Note that if $d_i$ is
odd, then $b^v_i(x) = 0$ for all infinite places $v \in \Omega_k$.

\medskip
We say that $x = (x_i^v) \in\underset{v\in\Omega_k}{\prod}   X_c(k_v)$ is a {\it local point} of $X_c$ if for each $i \in \cI$, we have $b_i^v(x) = 0$ for almost all $v \in \Omega_k$.
The following lemma implies the existence of local points whenever $\underset{v\in\Omega_k}{\prod}   X_c(k_v) \not =  \emptyset$.

\medskip

\begin{lemma}\label{almost_zero} Assume that $\underset{v\in\Omega_k}{\prod}   X_c(k_v) \not =  \emptyset$. Then there exists $$x = (x_i^v) \in\underset{v\in\Omega_k}{\prod}   X_c(k_v)$$ such that $b_i^v(x) = 0$ for almost all $v \in \Omega_k$, and for all $i \in \cI$.

\end{lemma}
\begin{proof}
For each $v\in\Omega_k$ such that $(K,c)^v$ is split, there exists $x^v_0\in K^v$ such that $N_{K^v/k_v}(x_0^v) = c$.  Then
$x = (x^v_0,1,...,1)$ is a $k_v$-point of $X_c$,  and
$b_i^v (x) =0$ for all $i \in \cI$.
Since $(K,c)^v$ is split for almost all places $v\in\Omega_k$, the lemma follows.
\end{proof}

\medskip
We now prove some properties of local points which will be used later. The next lemma is an analog of the classical reciprocity
formula that appears in the classical Brauer-Manin obstruction.

\medskip

\begin{lemma}\label{well_def}
{\it Let  $x = (x_i^v) \in\underset{v\in\Omega_k}{\prod}   X_c(k_v)$ be a local point of $X_c$.
Then we have  $$\underset{v\in\Omega_k}{\sum}\underset{i \in \cI}{{\sum}} b^v_i(x)=0.$$}
\end{lemma}
\medskip
\noindent
{\bf Proof.} Let us write  $x^v = (x_0^v,x_1^v,\dots,x_m^v)$, with $x_0^v \in K^v$ and $x_i^v \in K_i^v$ for $i \in \cI$. For all $i \in \cI$, set  $y_i^v=N_{L_i^v/k_v}(x_i^v)$.
Set  $y_0^v = N_{K^v/k_v}(x^v_0)$, and note that $y_0^v \underset{i \in \cI}{\prod} y_i^v=c$. We have
 \begin{center}
 $\underset{i \in \cI}{\sum} b^v_i(x) =\underset{i \in \cI}{\sum}{\rm inv}(K^v,y_i^v)= {\rm inv}(K^v, \underset{i \in \cI} {\prod} y_i^v) =
{\rm inv}(K^v,c/y_0^v)= {\rm inv}(K^v,c).$
 \end{center}
 Since $c\in k^\times,$ the Brauer-Hasse-Noether Theorem implies that  $\underset{v\in\Omega_k}{\sum} {\rm inv}(K^v,c) =0$. Hence we have
 $\underset{v\in\Omega_k}{\sum}\underset{i \in \cI}{{\sum}} b^v_i(x)=0$, as claimed.

 \medskip

\begin{lemma} \label{modify_condition}
{\it Let $x = (x_i^v)$ be a local point of $X_c$, and set $b_i^v= b_i^v(x) = {\rm inv}(K^v, N_{K_i^v/k_v}(x_i^v))$.
For all $i \in \cI$, let $\tilde{x}_i^v\in K_i^v$ and  set $$\tilde{b}_i^v= {\rm inv}(K^v, N_{K_i^v/k_v}(\tilde{x}_i^v)).$$
Suppose that for all $i \in \cI$ we have $\tilde b_i^v=0$ for almost all $v\in\Omega_k$, and that
$\underset{i \in \cI}{{\sum}}b_i^v=\underset{i \in \cI}{{\sum}}\tilde {b}_i^v$ for all $v\in\Omega_k$. Then for all $v \in \Omega_k$,
there exists  $\tilde{x}_0^v\in K^v$ such that $\tilde x = (\tilde x_i^v)$ is a local point of $X_c$. }
%$(\tilde{x}_0^v,...,\tilde{x}_m^v)$ is  a $k_v$-point of $ X_c$.}
\end{lemma}

\medskip
\noindent
{\bf Proof.}
Let $\tilde{y}_i^v= N_{K_i^v/k_v}(\tilde{x}_i^v)$ and $y_i^v= N_{K_i^v/k_v}(x_i^v)$.
Since  $\underset{i \in \cI}{{\sum}}b_i^v=\underset{i \in \cI}{{\sum}}\tilde {b}_i^v$, the algebras $(K^v,\underset{i \in \cI}{{\prod}}y_i^v)$
and  $(K^v,\underset{i \in \cI}{{\prod}}\tilde y_i^v)$ are isomorphic, hence
there exists some $z\in K^v$ such that
$(\underset{i \in \cI}{{\prod}}y_i^v)(\underset{i \in \cI}{{\prod}}\tilde{y}_i^v)^{-1}= N_{K^v/k_v}(z)$.
Therefore $(x_0^v z,\tilde{x}_1^v,...,\tilde{x}_m^v)$ is a $k_v$-point of $ X_c$.

\medskip
\noindent
\begin{lemma}\label{global_point_cond}
{\it Let $x = (x_i^v)$ be a local point of $X_c$, and set $b_i^v= b_i^v(x) = {\rm inv}(K^v, N_{K_i^v/k_v}(x_i^v))$.
Suppose that for all $i \in \cI$, we have  $\underset{v\in\Omega_k}{\sum} b_i^v=0$.
Then $ X_c$ has a $k$-point.}
\end{lemma}
\noindent
{\bf Proof.}
By the Brauer-Hasse-Noether Theorem, for every $i \in \cI$ there exists  a central simple algebra $A_i$ over $k$ such that ${\rm inv} (A_i ) = b_i^v$ for all $v \in \Omega_k$.
Set $y_i^v= N_{K_i^v/k_v}(x_i^v)$.
Since $(K^v,y_i^v)$ splits over $K^v$ for all $v$, the algebra $A_i$ also splits over $K$.
 Hence there exists  $\tilde{y}_i\in k$ such that $A_i$ is Brauer equivalent to $(K,\tilde{y}_i)$ (see \cite{GS} Cor. 4.7.6).
 Since  $(K,\underset{i \in \cI}{{\prod}}\tilde{y}_i)_v \simeq (K^v,\underset{i \in \cI}{{\prod}}y_i^v) \simeq (K,c)_v$, the Brauer-Hasse-Noether Theorem implies that
 $(K,\underset{i \in \cI}{{\prod}}\tilde{y}_i)  \simeq (K,c)$, and hence
 $\underset{i \in \cI}{{\prod}}\tilde{y}_i=c N_{K/k}(w)$ for some $w\in K^{\times}$.
 Moreover, we claim that the element $\tilde{y}_i$  belongs to the  group $ N_{K/k}(K^{\times}) N_{K_i/k}(K_i^{\times})$. To see this, we note that $$(K,\tilde{y}_i)_v=(K,y_i^v)=(K, N_{K_i^v/k_v}(x_i^v)).$$ Hence we have $\tilde y_i\in N_{K/k}(J_K) N_{K_{i}/k}(J_i)$  where $J_i$ is the id\`ele group of $K_i$, for all $i \in \cI$, and $J_K$ is the id\`ele group of $K$.
By Proposition \ref{hurlimann},  we have $\tilde{y}_i= N_{K/k}(w_i) N_{K_{i}/k}(z_i)$ for some $w_i\in K^{\times}$ and $z_i\in K_i^{\times}$.
Therefore $\underset{i \in \cI}{{\prod}}\tilde{y}_i=\underset{i \in \cI}{{\prod}} N_{K/k}(w_i) N_{K_i/k}(z_i)=c N_{K/k}(w)$
and  $(w^{-1}\underset{i \in \cI}{{\prod}}w_i,z_1,...,z_m)$ is a $k$-point of $ X_c$. This completes the proof of the lemma.

\medskip

\subsection{Brauer-Manin map - the prime power degree case}

Now suppose that $K$ is a cyclic extension of degree $d=p^e$, where $p$ is a prime.
Let $x=(x_i^v)\in\underset{v\in\Omega_k}{\prod} X_c(k_v)$ be a local point of $X_c$.
Let $M_i$ be a cyclic extension of $K_i$ such that the algebra $E_i$ is isomorphic to a product of copies of $M_i$; then the
degree of $M_i$ is $p^{e_i}$ for some $0\leq e_i \leq e$. Without loss of generality, we assume that $(e_1,...,e_m)$ is a decreasing sequence.
Let us define
\[
\alpha_c:\sha(K,K')  \to \rat/\ent
\]
by $\alpha_c(a_1,...,a_m)=\underset{v\in\Omega_k}{\sum}\underset{i \in \cI}{{\sum}} a_ib_i^v(x)$, where $(a_1,...,a_m)\in G \subseteq\underset{i \in \cI}{{\oplus}}\ent/p^{e_i}\ent$. Note that by Lemma \ref{split-inv}, we have
$b_i^v(x) \in\frac{1}{p^{e_i}}\ent/\ent$. Hence $a_ib_i^v(x)$ is well-defined. Moreover, by Lemma \ref{well_def}, the map $\alpha_c$ vanishes on the subgroup $D$ of $G$; hence, the map $\alpha_c : \sha (K,K') \to \rat / \ent$ is well-defined.

\medskip
Note that the map $\alpha_c$ is an analogue of the map given by the Brauer-Manin
pairing. In the following we show that $\alpha_c$ has a classical property of the Brauer-Manin pairing.

\medskip
\noindent
\begin{prop}\label{indep_pt}
{\it The map $\alpha_c :\sha(K,K')\ra \rat/\ent$  is independent of the choice of the local point $x = (x^v_i)$}.
\end{prop}

\medskip
\noindent
{\bf Proof.} We use the notation of section 5.1.
Let $a\in G$ and $I(a)=(I_0,....,I_{p^{e_1}-1})$.
If $a\in D$, then by Lemma \ref{well_def}, we have $\alpha_c(a)=0$.
In the following, we assume that $a\notin D$.

\medskip
 By the definition of $G$, we have  $ \Omega (I_0)\cup...\cup \Omega (I_{p^{e_1}-1})=\Omega_k$.
 Given a place $v\in\Omega_k$, there exists  $n(v)\in\ent/p^{e_1}\ent$ such that
 $v\in \Omega(I_{n(v)})$.
 Set $\delta_i = \delta(n(v),a_i)$ and let $K_i^v=\underset{w|v}{\prod} K_i^w$, where $K_i^w$ are field extensions of $k_v$.
 Then for all $i\notin I_{n(v)}$, the algebra $E_i^v$ is isomorphic to a products of field extensions of $K_i^w$ of degree at most $p^{\delta_i}$.
 Set $b_i^v = b_i^v(x)$;
 by Lemma \ref{split-inv}, we have $b_i^v\in\frac{1}{p^{\delta_i}}\ent/\ent$.
 By the definition of $\delta_i$, we have $\pi_{e_i,\delta_i}(a_i)=\pi_{e_1,\delta_i}(n(v))$. Hence for $i\notin I_{n(v)}$, we have
 \[a_ib_i^v=\pi_{e_i,\delta_i}(a_i)b_i^v=\pi_{e_1,\delta_i}(n(v))b_i^v=n(v)b_i^v.\]
Hence for all $v\in\Omega_k$, we have
 \begin{center}
 $\underset{i\in\cI}{\sum} a_i b_i^v=n(v)\underset{i\in \cI}{\sum} b_i^v=n(v){\rm inv}(K,c)_v$,
 \end{center}
 which is again independent of the $x^v_i$'s. Therefore, the map $\alpha_c$ is independent of the choice of the local point, and the proposition is proved.

 \medskip
 The map $\alpha_c:\sha(K,K')\ra \rat/\ent$  will be called  the \emph{Brauer-Manin map} for $X_c$.

 \medskip
Let $K_0$ be the unique subfield of $K$ such that $[K_0:k] = p^{e-1}$, and
set $L_0 = K_0 \times K'$. If $c \in k^{\times}$, let $X^0_c$ be the affine $k$-variety determined by $N_{L_0/k}(t) = c$. There is a natural map $\rho:X_c\to X^0_{c}$ defined as $\rho(x_0,...,x_m)=(N_{K/K_0}(x_0),x_1,...,x_m)$.

If $X^0_c(k_v) \not = \emptyset$
for all $v \in \Omega_k$, we denote by $\alpha^0_c : \sha(K_0,K') \to \rat/\ent$ the corresponding Brauer-Manin map.

\medskip
If $t_i \in K^v_i$,  set

\medskip

\centerline {$b_i^v(K,t_i) = {\rm inv}(K^v,N_{K_i^v/k_v}(t_i))$, and
$b_i^v(K_0,t_i) = {\rm inv}(K_0^v,N_{K_i^v/k_v}(t_i))$.}

\medskip
Recall that  a local point of $X_c$ is $x = (x_i^v) \in\underset{v\in\Omega_k}{\prod}   X_c(k_v)$ such that  for each $i \in \cI$, we have $b_i^v(K,x_i^v)  = 0$ for almost all $v \in \Omega_k$.

\medskip
The following lemma is an analogue of the functoriality of the Brauer-Manin pairing.
\begin{lemma}\label{alpha0}
Assume that $X_c(k_v) \not = \emptyset$ for all $v \in \Omega_k$. Then we have

\medskip {\rm (i}) $X_c^0(k_v) \not = \emptyset$ for all $v \in \Omega_k$.

\medskip {\rm (ii}) $\alpha_c \circ F = \alpha_c^0$.

\end{lemma}

\medskip
\noindent
{\bf Proof.} If $x^v \in X_c(k_v)$, then $N_{L^v/L_0^v}(x^v) \in X_c^0(k_v)$. This proves  {\rm (i}). Let us check  {\rm (ii}). Let $x = (x_i^v)$ be a local point of $X_c$.
Note that
$b^v_i (K_0,x^v_i) = p b^v_i (K,x^v_i) $. Let $a \in \sha(K_0,K')$. Then we have $$\alpha_c(F(a)) = \underset {v \in \Omega_k}\sum \underset {i \in \cI} \sum a_i (p b^v_i (K,x^v_i) ) =  \underset {v \in \Omega_k}\sum \underset {i \in \cI} \sum a_i b^v_i (K_0,x^v_i) = \alpha_c^0(a).
$$ This completes the proof of the lemma.

%%%%%%%%%%%%%%%%%%%%%%%%%%%%%%%%%%%%%%%%%%%%%%%%%%%%%%%%%%%%%%%%%

\medskip

\subsection{Brauer-Manin map - the general case}

Recall that $K/k$ is a cyclic extension of degree $d$, and that $L = K \times  K'$, where $K'$ is an \'etale $k$-algebra.
We keep the notation of 5.3, in particular, $\cP$ is the set of prime divisors of $d$.  For all $p \in \cP$, we denote by $K(p)$ the
largest subfield of $K$ of degree a power of $p$, and we set $L(p) = K(p) \times K'$.
For all $c\in k^{\times}$ and $p \in \cP$,  we let $X_c(p)$ be the $T_{L(p)/k}$-torsor defined by
\[
N_{L(p)/k}(x) =c.
\]
 Let $x=(x_i^v)\in\underset{v\in\Omega_k}{\prod} X_c(k_v)$ be a local point of $X_c$, and
let us write $x = (x^v_0,x'^v)$ with $x^v_0 \in K^v$ and $x'^v \in K'^v$. Then $(N_{K^v/K(p)^v}(x^v_0),x'^v )$ is a local point of $X_c(p)$.

\medskip
Let $\alpha (p)$  be the Brauer-Manin map of $X_c (p)$, as defined above. By Proposition \ref{indep_pt} the map $\alpha (p)$ is independent of the
choice of the local point. Recall that $\sha (K,K') =  \underset{p \in \cP}{\oplus}  \sha(K(p),K')$, and let us define $\alpha_c : \sha(K,K') \to \rat/\ent$
by $\alpha_c = \underset{p \in \cP} \oplus \alpha_c(p)$. Hence $\alpha_c$ is also independent of the choice of the local point.
We call $\alpha_c$ the \emph{Brauer-Manin map} for $X_c$.

\section{Necessary and sufficient condition}

\medskip We keep the notation of the previous sections.
The main theorem is the following:
\begin{theo} \label{main_theo}
{\it
The affine $k$-variety $ X_c$ has a $k$-point if and only if
$ X_c$ has a $k_v$-point at each place $v\in\Omega_k$ and $\alpha_c$ is the zero map.}
\end{theo}

\subsection{The prime power degree case}
We suppose that $K$ is cyclic of degree $p^e$, where $p$ is a prime number and $e\geq 1$. The proof of Theorem \ref{main_theo} uses induction on $e$.

\medskip
The proof can be divided into three parts.
Recall that $\rho : X_c \to X_c^0$ is the map given by $(x_0,x_1,\dots,x_m) \mapsto ({\rm N}_{K/K_0}(x_0),x_1,\dots,x_m).$ (See \S 6.2.)
Suppose that $X_c$ has a $k_v$-point for all $v\in\Omega_k$ and that $X_c^0$ has a $k$-point $z=(z_0,...,z_m)$.
In the first part we use the point $z$ to get a system of local solutions $(\tilde{x}_i^v)$ of $X_c$ such that
$\underset  {v \in \Omega_k} \sum b_i^v(K,\tilde x_i^v) \in {1 \over p}\ent /\ent$. (See Lemma \ref{times p}-Lemma \ref{complement}.)

In Lemma \ref{alpha bar}-Lemma \ref{step 3}, we show that one can further modify $(\tilde{x}_i^v)$ so that
$\underset  {v \in \Omega_k} \sum b_i^v(K,\tilde x_i^v)=0$.
In the end we conclude  our main theorem by induction on $e$ and Lemma \ref{global_point_cond}.

\medskip
We would like to mention that an alternative way to prove the theorem is to find the generators of
the unramified Brauer group of $X_c$ and show that the Brauer-Manin map defined here is the evaluation on the unramified Brauer group. However, here we choose to prove the theorem in a more elementary way and keep the unramified Brauer group for our future work.

\medskip
We start with some preliminary results.

\medskip
Recall that $E_i = K \otimes_k K_i$.

\begin{lemma}\label{local_modify}
Suppose that $K$ is a cyclic extension of degree $p^e$, where $p$ is a prime and $e\geq 1$.
Let $v\in\Omega_k$ and $i\in \cI$ be such that $E_i^v$ is not isomorphic to
a product of copies of $K^v_i$. Then for all  $b\in\frac{1}{p}\ent/\ent\subseteq\rat/\ent$, there exists $x\in K^v_i$ such that
${\rm inv}(K^v,N_{K^v_i/k_v}(x))=b$.
\end{lemma}
\begin{proof}
Suppose that $K^v$ is isomorphic to a product of copies of a field extension $M$ of $k_v$, and set
$[M:k_v] = p^ f$. Since by hypothesis $E_i^v$ is not isomorphic to a product of copies of $K^v_i$, we have $f \geq 1$. Assume that  $K_i^v \simeq \underset {j \in J} \prod M_{i,j}$,
where $M_{i,j}$ is a field extension of $k_v$ for all $j \in J$.

\medskip It suffices to prove that  $\frac{1}{p}\ent/\ent\subseteq\inv (M,N_{M_{i,j}/k_v}(M_{i,j}^{\times}))$ for some $j \in J$; hence we may assume
that $K^v$ is a field extension of $k_v$ of degree $p^e$ with $e\geq 1$,  and that $K_i^v$ is a field.

\medskip
Let  $\Br(K^v/ k_v)$ be the subgroup of the Brauer group of $k_v$ split by $K^v$; this
group is isomorphic to $\ent/p^e\ent \simeq k_v^\times/N((K^v)^\times$). (See \cite{GS} Cor. 4.4.10 and \cite{CF} Chap. VI \S1.1 Thm. 3 Cor. 2.)

\medskip
For all $i \in \cI$, let $M_i$ be a field such that $E_i^v = K \otimes_k K^v_i$ is a product of copies of  $M_i$, and set $[M_i :K^v_i] = p^{e^v_i}$; the hypothesis implies that
$e^v_i \ge 1$.
The corestriction map  $\Br(K^v_i) \to \Br(k_v)$ is an injection and restricts
to an injection of $\Br(M_i/K^v_i)$ into $\Br(K^v/ k_v)$, the image being the
unique subgroup of order $p^{e^v_i}$ of the cyclic group of order $p^e$.
By the projection formula (\cite{GS} Prop. 3.4.10), the image consists of cyclic algebras of the type $(K^v, N_{K^v_i/k_v}(z))$ with
$z$ an element of $K^v_i$.  Hence $\frac{1}{p}\ent/\ent\subseteq\frac{1}{p^{e^v_i}}\ent/\ent=\inv(K^v, N_{K^v_i/k_v}(K^v_i)^\times)$.
This completes the proof of the lemma.
\end{proof}

%%%%%%%%%%%%%%%%%%%%%%%%%%%%%%%%%%%%%%%%%%%%%%%%%%%%%%%%%%%%%%%%%%%%%%%%%%%%%%%%%%%%%%%%%%%

\medskip
Let $K_0$ be the unique subfield of $K$ such that $[K_0:k] = p^{e-1}$.

\medskip
Recall that we have $K' =  \underset{i \in \cI} \prod K_i$, that $E_i = K \otimes K_i$, and that $E_i$ is the product of copies of a cyclic extension of
degree $p^e_i$ of $K_i$. Set $E_i^0 = K_0 \otimes K_i$. Then $E_i^0$ also splits as a product of copies of a cyclic extension of $K_i$; let us denote by $p^f_i$
the degree of this extension. Recall that for all $i \in \cI$, we have $f_i \le e_i$. If moreover $e_i \not = 0$, then $e_i = f_i + 1$ (cf. lemma \ref{e and f}).
For all $i \in \cI$, the map $F_i : \ent/p^{f_i} \ent \to \ent/ p^{e_i} \ent$ is the inclusion
of the unique subgroup of order $p^{f_i}$ in the group  $\ent/ p^{e_i} \ent$, and we set $F = \underset{i \in \cI} \oplus F_i$.

\medskip
The map $F : \underset {i \in \cI} \oplus \ent/p^{f_i} \ent \to  \underset {i \in \cI} \oplus \ent/ p^{e_i} \ent$  induces a
homomorphism $F : \sha(K_0,K') \to \sha (K,K')$.
Recall that the cokernel of $F$ is
isomorphic to the group $\sha(K/K_0,K')$, defined in section 5.

\medskip
For all $i \in \cI$, set $r_i = {\rm min} \{1,e_i \}$. For all  $c  \in \underset {i \in \cI} \oplus \ent/p^{r_i}\ent$ and $n \in \ent / p \ent$, set $I^1_n(c) = \{ i \in \cI \ \ | n \succeq  c \}$.
If $I^1_n(c) \not = \cI$, set $\Omega (I^1_n(c)) = \underset{i\notin I^1_n(c)}{\cap}{\Sigma}_i$; if $I^n(c) = \cI$, set $\Omega (I^1_n(c)) = \Omega_k$. Set

\[G(K/K_0,K') =\{c \in \underset {i \in \cI} \oplus \ent/p^{r_i}\ent  \ \ |    \ \underset{n\in\ent/p^{r_1}\ent}{\bigcup} \Omega (I_n(c))=\Omega_k\},
\]
let $D(K/K_0,K')$ be the diagonal subgroup of $G(K/K_0,K')$, and recall that $$\sha(K/K_0,K') = G(K/K_0,K') / D(K/K_0,K').$$

\medskip
Recall that the projection $\pi : \underset {i \in \cI} \oplus \ent/p^{e_i} \ent \to  \underset {i \in \cI} \oplus \ent/ p^{r_i} \ent$  induces a
homomorphism $F : \sha(K,K') \to \sha (K/K_0,K')$ (cf. lemma \ref{K/K_0})

\medskip

If $t_i \in K^v_i$,  set $b_i^v(K,t_i) = {\rm inv}(K^v,N_{K_i^v/k_v}(t_i))$, and
$b_i^v(K_0,t_i) = {\rm inv}(K_0^v,N_{K_i^v/k_v}(t_i))$.

\medskip
Recall that  a local point of $X_c$ is $x = (x_i^v) \in\underset{v\in\Omega_k}{\prod}   X_c(k_v)$ such that for each $i \in \cI$, we have $b_i^v(K,x_i^v)  = 0$ for almost all $v \in \Omega_k$.

\medskip
\begin{lemma}\label{times p} Let $x = (x_i^v)$ be a local point of $X_c$, and let $z = (z_i)$ be a global point of $X_c^0$. Then for all
$v \in \Omega_k$, we have

$$p \  \underset {i \in \cI}  \sum b^v_i(K,x_i^v) = p \  \underset {i \in \cI}  \sum b_i^v(K,z_i).$$

\end{lemma}

\medskip
\noindent
{\bf Proof.} Since $x$ is a local point of $X_c$, we have $ \underset {i \in \cI}  \sum b_i^v(K,x_i^v) = {\rm inv}(K^v,c)$ for all $v \in \Omega_k$.
Similarly, we have $ \underset {i \in \cI}  \sum b_i^v(K_0,z_i)=  {\rm inv}(K_0^v,c)$ for all $v \in \Omega_k$. Note that ${\rm inv}(K_0^v,c) =
p \ {\rm inv}(K^v,c)$, and ${\rm inv}(K_0^v,z_i) =p \ {\rm inv}(K^v,z_i)$ for all $i \in \cI$. Hence we have
$$p \  \underset {i \in \cI}  \sum b_i^v(K,x_i^v) =  p \ {\rm inv}(K^v,c) = {\rm inv}(K_0^v,c) =  \underset {i \in \cI}  \sum b_i^v(K_0,z_i) = p \  \underset {i \in \cI}  \sum b_i^v(K,z_i),$$
as claimed.

\medskip
\begin{lemma}\label{three parts} Let $x = (x_i^v)$ be a local point of $X_c$, and assume that $X_c^0(k) \not = \emptyset$. Then there
exist  $\tilde x_i^v \in K_i^v$ such that

\medskip
{\rm (i)} For each $i \in \cI$, we have $b_i^v(K,\tilde x_i^v) = 0$ for almost all $v \in \Omega_k$.

\medskip
{\rm (ii)} For all $i \in \cI$, we have $\underset  {v \in \Omega_k} \sum b_i^v(K,\tilde x_i^v) \in {1 \over p}\ent /\ent$.

\medskip
{\rm (iii)} For all
$v \in \Omega_k$, we have

$$ \underset {i \in \cI}  \sum b_i^v(K,x_i^v) =   \underset {i \in \cI} \sum b_i^v(K,\tilde x_i^v).$$

\end{lemma}

\medskip
\noindent
{\bf Proof.} Let  $z = (z_i)$ be a global point of $X_c^0$. Set $b_i^v = b_i^v(K,x_i^v)$ and $h_i^v = b_i^v(K,z_i)$. By Lemma \ref{times p}, we have
$p\underset{i\in\cI}{\sum} b_{i}^v = p\underset{i\in\cI}{\sum} h_{i}^v$.
Since  $b_i^v=0$ and $h_i^v = 0$ for almost all $v \in \Omega_k$, for almost places $v \in \Omega_k$ we have $\underset{i\in\cI}{\sum} h_{i}^v=\underset{i\in\cI}{\sum} b_{i}^v$.
Suppose that there is $v \in \Omega_k$ such that $\underset{i\in\cI}{\sum} h_{i}^v \neq\underset{i\in\cI}{\sum} b_{i}^v$.
If $v\in\underset{i\in\cI}{\cap}\Sigma_i$, then $b_i^v=h_i^v=0$ for all $i\in\cI$, which is a contradiction.
Hence there exists  $i \in \cI$ such that $v\notin\Sigma_i$.
Since $p\underset{j \in \cI}{\sum} b_{j}^v=p\underset{j \in \cI}{\sum} h_{j}^v$ in $\rat/\ent$, we know that $\underset{j\in\cI}{\sum}b_{j}^v-\underset{j\in\cI}{\sum} h_{j}^v\in \frac{1}{p}\ent/\ent$. By Lemma \ref{local_modify}, there exists $\tilde x_i^v \in K_i^v$ such that $\inv(K^v,N_{K_i^v/k_v}({\tilde x}_i^v))=h_i^v-\underset{j\in\cI}{\sum} (h_{j}^v-b_{j}^v)$.

\medskip
Set $\tilde h_i^v = \inv(K^v,N_{K_i^v/k_v}({\tilde x}_i^v))$; for all $j \not = i$, let $\tilde x_j^v = z_j$, $\tilde h_j^v = h_j^v = b_j^v(K,z_j)$. Then we have  $\underset{j\in\cI}{\sum} \tilde h_{j}^v=\underset{j\in\cI}{\sum} b_{j}^v$; this proves {\rm (iii)}.

\medskip
Since $\tilde h_i^v = h_i^v$ for almost all $v \in \Omega_k$, {\rm (i)} holds.
As $z = (z_i)$ is a global point of $X_c^0$, we have
$\underset  {v \in \Omega_k} \sum b_i^v(K_0,z_i) = 0$,
hence $\underset  {v \in \Omega_k} \sum h_i^v  \in {1 \over p}\ent /\ent$; moreover, $h_i^v - \tilde h_i^v \in {1 \over p}\ent /\ent$
for all $i \in \cI$ and all $v \in \Omega_k$. Therefore we have $\underset  {v \in \Omega_k} \sum \tilde h_i^v \in {1 \over p}\ent /\ent$, and this proves
{\rm (ii)}.

\medskip
\begin{lemma}\label{complement} Assume that $X_c(k_v) \not = \emptyset$ for all $v \in \Omega_k$, and that $X_c^0(k)  \not = \emptyset$. Then there
exists a local point $\tilde x = (\tilde x_i^v)$ of $X_c$ such that for all $i \in \cI$, we have $$\underset  {v \in \Omega_k} \sum b_i^v(K,\tilde x_i^v) \in {1 \over p}\ent /\ent.$$

\end{lemma}

\medskip
\noindent
{\bf Proof.} Let $x = (x_i^v)$ be a local point of $X_c$. By Lemma \ref{three parts}, there exist  $\tilde x_i^v \in K_i^v$ such that  $b_i^v(K, \tilde x_i^v) = 0$ for almost all $v \in \Omega_k$, that $ \underset {i \in \cI}  \sum b_i^v(K,x_i^v) =   \underset {i \in \cI} \sum b_i^v(K,\tilde x_i^v)$, and that
for all $i \in \cI$, we have $\underset  {v \in \Omega_k} \sum b_i^v(K,\tilde x_i^v) \in {1 \over p}\ent /\ent$.
By Lemma \ref{modify_condition}, for all $v\in\Omega_k$, there exists $\tilde x^v_0\in (K^v)^{\times}$ such that $(\tilde x^v_0,\tilde x^v_1,...,\tilde x^v_m)\in X_c(k_v).$
This completes the proof of the lemma.

\bigskip
Recall that if $X_c(k_v) \not =\emptyset$ for all $v \in \Omega_k$, and that for all $c \in k^{\times}$, we have a homomorpism $\alpha_c : \sha(K,K') \to \rat/\ent$. We now
show that $\alpha_c$ induces  a homomorphism $\overline \alpha_c : \sha (K/K_0,K') \to \rat /\ent$ such that $\overline \alpha_c \circ \pi = \alpha_c$.

\medskip
\begin{lemma}\label{alpha bar}
Assume that $X_c(k_v) \not =\emptyset$ for all $v \in \Omega_k$
and that $X_c^0(k)  \not = \emptyset$.
Then there exists a
homomorphism $\overline \alpha_c : \sha (K/K_0,K') \to \rat /\ent$ such that $\overline \alpha_c \circ \pi = \alpha_c$.

\end{lemma}

\medskip
\noindent
{\bf Proof.} Let $x = (x_i^v)$ be a local point of $X_c$, and set $b_i^v = \inv(K^v,N_{K_i^v/k_v}(x_i^v))$
for all $i \in \cI$ and all $v \in \Omega_k$.
Let $a=(a_1,...,a_m)\in G$. Then by Lemma \ref{K/K_0} $\pi(a)=(\ol{a}_1,...,\ol{a}_{m})\in G(K/K_0,K')\subseteq\underset {i \in \cI} \oplus \ent/ p^{r_i} \ent$. Note that $r_i=\min\{ 1,e_i\}$ by definition.
If $e_i=0$, then $e_i=r_i=0$ and $a_i=\ol{a}_i=0$. If not, then $r_i=1$ and $\ol{a}_i=a$ (mod $p$).
By lemma \ref{complement}, we may assume that $\underset  {v \in \Omega_k} \sum b_i^v \in {1 \over p}\ent /\ent$ for all $i \in \cI$; hence $a_i(\underset{v\in\Omega_k}{\sum}b^v_i)=\ol{a}_i(\underset{v\in\Omega_k}{\sum}b^v_i)$.
We then have \[\alpha_c(a_1,...,a_m)=\underset{i\in\cI}\sum a_i(\underset{v\in\Omega_k}{\sum}b^v_i)=\underset{i\in\cI}\sum \ol{a}_i(\underset{v\in\Omega_k}{\sum}b^v_i).\]
Hence  $\alpha_c$ induces a homomorphism $\overline \alpha_c : \sha (K/K_0,K') \to \rat /\ent$, as claimed.

\medskip
\begin{lemma}\label{step 1}
Let $x = (x_i^v)$ be a {\it local point} of $X_c$, and set $b^v_i = b^v_i(K,x_i^v)$. Assume that $\alpha_c = 0$, and that $X^0_c(k) \not = \emptyset$. Let $(I_0,\dots,I_{p-1}) \in G(K/K_0,K')$.
Then we have $\underset{i\in I_n}\sum\ \underset{v\in\Omega_k}{\sum}b_i^v=0$ for all  $n\in\ent/p\ent$.
\end{lemma}

{\bf Proof.}
Let $n \in \ent/p\ent$. The statement is trivial if $I_n$ is empty, and it follows from Lemma \ref{well_def} if $I_n = \cI'$. Assume that $I_n$ is not empty, and $I_n \not = \cI'$.  Let
$n' \in \ent/p\ent$ such that $n' \not = n$, and set  $J_n = I_n$, $J_{n'} = \underset {r \not = n} \cup I_r$, and
$J_r = \emptyset$ if $r \not = n,n'$. Then by Lemma \ref{simpl}, we have $(J_0,\dots,J_{p-1}) \in G(K/K_0,K')$. Since $X^0_c(k) \not = \emptyset$, by Lemma \ref{alpha bar}
there exists a
homomorphism $\overline \alpha_c : \sha (K/K_0,K') \to \rat /\ent$ such that $\overline \alpha_c \circ \pi = \alpha_c$.
By hypothesis $\alpha_c$ is the zero map, hence we have $\overline \alpha_c =  0$. Therefore we have $$\underset{r\in\ent/p\ent}{\sum}\ \underset{i\in J_r}\sum\ \underset{v\in\Omega_k}{\sum}rb_i^v=\underset{i\in J_{n}}\sum\ \underset{v\in\Omega_k}{\sum}n b_i^v+\underset{i\in J_{n'}}\sum\ \underset{v\in\Omega_k}{\sum}n' b_i^v=0.$$
By Lemma \ref{well_def} we have  $\underset{i\in\cI'}\sum\ \underset{v\in\Omega_k}{\sum}b_i^v=0$, hence $(n-n')\underset{i\in J_{n}}\sum\ \underset{v\in\Omega_k}{\sum}b_i^v=0$. Recall that $n' \not = n$ by hypothesis, therefore we have
$\underset{i\in J_{n}}{\sum}\underset{v\in\Omega_k}{\sum}b_i^v=0$; since $J_n = I_n$,
we have $\underset{i\in I_{n}}\sum\ \underset{v\in\Omega_k}{\sum}b_i^v=0$, as claimed.

\medskip
\begin{lemma}\label{step 2}
Let $x = (x_i^v)$ be a {\it local point} of $X_c$.  Assume that $\alpha_c = 0$, and that $X^0_c(k) \not = \emptyset$. Let $(I_0,\dots,I_{p-1}) \in G(K/K_0,K')$, and let $n\in\ent/p\ent$. Then there exists a local point $\tilde x = (\tilde x_i^v)$ of $X_c$ such that $\tilde x^v_i = x^v_i$ if $i \not \in I_n$,  and that
%$\underset{i\in \cI}{\sum} b^v_{i}(K,\tilde x) = \underset{i\in \cI}{\sum} b^v_{i}(K,x)$, and that
$\underset{v\in\Omega_k}{\sum} {b}_i^v (K,\tilde x) =0$ for all $i\in I_n.$

\end{lemma}

\noindent
{\bf Proof.} We prove this by induction on the cardinality of $I_n.$
If $|I_n|=0$ then the claim is trivial; if $|I_n| = 1$, then it follows from  lemma \ref{step 1}, since  we have  $\underset{v\in\Omega_k
}{\sum}b_i^v(x)=0$ for all $i\in I_n.$
Suppose that the claim is true for $|I_n|<h$. For $|I_n|=h$, suppose that there are nonempty disjoint subsets $I^0_n$ and $I^1_n$ of $I_n$ satisfying
$I_n^0\cup I_n^1=I_n$ and $(\underset{i\in I_n^0}{\cap}\Sigma_i)\cup(\underset{i\in I_n^1}{\cap}\Sigma_i)=\Omega_k.$
Then consider the element $( J_0,... J_{p-1})$ where $J_r = I_r$ if $r \not = n, n+1$, $J_n = I_n^0$ and $J_{n+1} = I_n^1 \cup I_{n+1}$.
Note that $\Omega(I_r) = \Omega(J_r)$ if $r \not = n, n+1$ and that  $\Omega (I_{n+1}) \subset \Omega(J_{n+1})$.
Let us prove that $( J_0,... J_{p-1})$ represents an element of $\sha(K,K')$; for this, we have to check that $\Omega(J_0) \cup \dots \cup \Omega(J_{p-1}) =
\Omega_k$. Since $\Omega(I_r) \subset \Omega(J_r)$ if $r \not = n$ and $\Omega(I_0) \cup \dots \cup \Omega(I_{p-1}) = \Omega_k$, it suffices
to check that if $v \in \Omega (I_n)$, then $v \in \Omega(J_0) \cup \dots \cup \Omega(J_{p-1})$. If $v \in \underset {i \in I_n^1} \cap \Sigma_i$,
then we have $v \in \Omega (J_n)$. Otherwise, we have $v\in\underset{i\in I_n^0}{\cap}\Sigma_i$ because  $(\underset{i\in I_n^0}{\cap}\Sigma_i)\cup(\underset{i\in I_n^1}{\cap}\Sigma_i)=\Omega_k.$
Hence we have $v\in(\underset{i\notin I_n\cup I_{n+1}}{\cap}\Sigma_i)\cap(\underset{i\in I_n^0}{\cap}\Sigma_i)=  \Omega (J_{n+1}).$
Therefore $( J_0,... J_{p-1})$ represents an element of $\sha(K,K')$. Since $|J_n|  < h$, we can apply the induction hypothesis, and hence there exists a local point
$\tilde x = (\tilde x_i^v)$ such that $\tilde x^v_i = x^v_i$ if $i \not \in J_n = I_n^1$, and that $\underset{v\in\Omega}{\sum} {b}_i^v (K,\tilde x) =0$ for all $i\in J_n = I^1_n.$ The same argument with $I_n^0$ instead of $I_n^1$ gives the desired result.

\smallskip

Assume now that $I_n$ does not have any non-trivial subpartitions,  in other words, that  there are no nonempty disjoint subsets $I^0_n$ and $I^1_n$ of $I_n$ satisfying
$I_n^0\cup I_n^1=I_n$ and $(\underset{i\in I_n^0}{\cap}\Sigma_i)\cup(\underset{i\in I_n^1}{\cap}\Sigma_i)=\Omega_k.$ Let us consider the
graph with vertex set $I_n$, and edge set $\cE=\{(i,j)|\Sigma_i\cup\Sigma_j\neq\Omega_k\}$; since $I_n$ has no non-trivial subpartitions, this
graph is connected. Set $b_i^v = b(K,x_i^v)_i^v$, and for all $i \in I_n$, set $d_i=\underset{v\in\Omega_k}{\sum} b^v_{i}$. Let us fix an ordering of $I_n$, say $I_n = \{i_0,...,i_t\}$.
Since the graph is connected, there exists a loop-free path between $i_0$ and $i_1$. Along this path,
for any two adjacent vertices $i$, $j$,  there exists $v\in\Omega_k$ such that $v \not \in \Sigma_{i}\cup\Sigma_{j}$. By Lemma \ref{complement} we may
assume that $b_i^v \in {1 \over p} \ent / \ent$ for all $i \in I_n$.
Applying Lemma \ref{local_modify}, by modifying $x_i^v$ and $x_j^v$ we can modify $b^v_{i}$ to  $b^v_{i}-d_{i_0}$ and $b^v_{j}$ to  $b^v_{j}+d_{i_0}$.
Note that this modification does not change $\underset{i\in \cI}{\sum} b^v_{i}.$
Therefore by Lemma \ref{modify_condition}, after changing also $x_0^v$ if necessary, the modified  $(x_i^v)$ is still a local point of $X_c$.
After these modifications, we have $\underset{v\in\Omega_k}{\sum} b^v_{i_0}=0$, $\underset{v\in\Omega_k}{\sum} b^v_{i_1}=d_{i_1}+d_{i_0},$
and all the other $d_i$'s remain unchanged.
We repeat this process along a loop-free path from $i_1$ to $i_2$, and
we modify each adjacent pair along the path from $i_1$ to $i_2$ by $d_{i_0}+d_{i_1}$ and so on.
At the end, we modify each adjacent pair along the path from $i_{t-1}$ to $i_t$ by $\underset{r=0}{\overset{t-1}{\sum}}d_{i_r}$.
After this process, we have $\underset{v\in\Omega_k}{\sum} b^v_{i_r}=0$  for $r=0,...,t-1$ and $\underset{v\in\Omega_k}{\sum} b^v_{i_t}=d_{i_t}+\underset{r=0}{\overset{t-1}{\sum}}d_{i_r}.$ However, by Lemma \ref{step 1}, we know that $\underset{r=0}{\overset{t}{\sum}}d_{i_r}=0$; hence,
we have $\underset{v\in\Omega_k}{\sum} b^v_{i_t}=0.$
Moreover, only finitely many $b_i^v$'s are modified, so $b_i^v=0$ for almost all $v$;
the lemma then follows.

\medskip
\begin{prop}\label{step 3}

Let $x = (x_i^v)$ be a {\it local point} of $X_c$. Assume that $\alpha_c = 0$, and that $X_c^0(k)  \not = \emptyset$.
Then there exists a local point $\tilde x = (\tilde x_i^v)$ of $X_c$ such that
%$\underset{i\in \cI}{\sum} b^v_{i}(K,\tilde x) = \underset{i\in \cI}{\sum} b^v_{i}(K,x)$, and that
for all $i \in \cI$, we have $$\underset  {v \in \Omega_k} \sum b_i(K,\tilde x_i^v) = 0.$$

\end{prop}

\medskip
\noindent
{\bf Proof.} This follows from Lemma \ref{step 2}.

\bigskip
\noindent
{\bf Proof of Theorem  \ref{main_theo} for $K$ of prime power degree.}

\medskip
It is clear that if $X_c$ has a $k$-point, then $X_c$ has a $k_v$-point for all $v\in\Omega_k$ and $\alpha_c=0$.
Conversely, suppose that $X_c$ has a $k_v$-point for all $v\in\Omega_k$ and that $\alpha_c=0$. Let us show that
the variety $X_c$ has a $k$-point.
We  show our claim by induction on the exponent $e$.  Suppose that
$e=1$. Then $K_0 = k$, and $X_c^0 (k) \not = \emptyset$. By Proposition \ref{step 3},
there exists a local point $x = (x_i^v)$ of $X_c$ such that
for all $i \in \cI$, we have $\underset  {v \in \Omega_k} \sum b_i( x_i^v) = 0.$
Lemma \ref{global_point_cond} implies that $X_c(k) \not = \emptyset$.
Assume now that $e > 1$.
Since $X_c$ has a $k_v$-point for all $v\in\Omega_k$, the variety $X_c^0$ also has a $k_v$-point for all $v\in\Omega_k$.
As $\alpha_c$ is the zero map, by Lemma \ref{alpha0} the Brauer-Manin map $\alpha_c^0$ for $X_c^0$ is also the zero map.
Therefore $X_c^0$ has a $k$-point by induction hypothesis. By Proposition \ref{step 3},  there exists a local point $x = (x_i^v)$ of $X_c$ such that
for all $i \in \cI$, we have $\underset  {v \in \Omega_k} \sum b_i( x_i^v) = 0.$
Lemma \ref{global_point_cond} implies that $X_c(k) \not = \emptyset$.

\bigskip

\subsection{The general case}

\medskip

Recall that $K/k$ is a cyclic extension of degree $d$, and that $L = K \times  K'$, where $K'$ is an arbitrary \'etale $k$-algebra.
We keep the notation of 5.3, in particular, $\cP$ is the set of prime divisors of $d$.  For all $p \in \cP$, we denote by $K(p)$ the
largest subfield of $K$ of order a power of $p$, and  $L(p) = K(p) \times K'$.
For all $c\in k^{\times}$ and $p \in \cP$,   the affine $k$-variety defined by
\[
N_{L(p)/K(p)}(x) =c.
\]
is denoted by $X_c(p)$. Recall that there is a natural map from $X_c$ to $X_c(p)$ (\S 6.3). We denote by $\alpha_c (p)$  be the Brauer-Manin map of $X_c (p)$.  Recall that $\sha (K,K') =  \underset{p \in \cP}{\oplus}  \sha(K(p),K')$, and that  $\alpha_c : \sha(K,K') \to \rat/\ent$ is given
by $\alpha_c = \underset{p \in \cP} \oplus \alpha_c(p)$.

\medskip

\begin{lemma}\label{coprime pt}
Let $c\in k^{\times}$. Then $X_c$ has a $k$-point if and only if $X_c(p)$  has a $k$-point for all $p \in \cP$.
\end{lemma}

\medskip
\noindent
{\bf Proof.} Let $z \in X_c(k)$ be a $k$-point of $X_c$, and
let us write $z = (x,y)$ with $x \in K$ and $y  \in K'$. Then $(N_{K/K(p)}(x),y )$ is a $k$-point of $X_c(p)$ for all $p \in \cP$.
Conversely, suppose that for all $p \in \cP$, the $k$-variety $X_c(p)$ has a $k$-point $(x_{p},y_p)  \in K(p) \times K'$.
For all $p \in \cP$, set  $$r_p = \underset {q \in \cP, q \not = p} \prod [K(q):k],$$
and let $s_p \in \ent$ such that $\underset {p \in \cP} \sum r_p s_p = 1.$ Set $x = \underset {p \in \cP} \prod x_{p}^{s_p}$, and
$y =  \underset {p \in \cP} \prod y_p^{r_p s_p}$. Then $(x,y)$ is a $k$-point of $X_c$.

\medskip
\begin{remark}
The above lemma is compatible with base change to any field extension $l$ of $k$.
\end{remark}

\medskip
\noindent
{\bf Proof of Theorem \ref{main_theo}.}
Suppose that $X_c$ has a $k$-point. Then by lemma \ref{coprime pt}, $X_c(p)$ has a $k$-point for all  $p \in \cP$. This implies that  $\alpha_c (p) = 0$ for
all $p \in \cP$, and hence $\alpha_c=0$.
Conversely, suppose that $X_c$ has a $k_v$-point for all $v \in \Omega_k$ and that $\alpha_c=0$.
Then $X_c(p)$  has a $k_v$-point for all $v \in \Omega_k$.
Since $\alpha_c = 0$, we have $\alpha_c (p) = 0$ for
all $p \in \cP$. But $K(p)$ is a cyclic extension of prime power degree, hence this implies that $X_c(p)$ has a $k$-point for all  $p \in \cP$.
Therefore $X_c$ has a $k$-point by Lemma \ref{coprime pt}.

\begin{coro} \label{idele} Let $I_L$ be the id\`ele group of $L$. Then sending $c \in k^{\times}$ to $\alpha_c$ gives rise to an isomorphism $$(k^{\times} \cap N_{L/k}(I_L)) / N_{L/k}(L^{\times}) \to \sha(L)^*.$$

\end{coro}

\medskip
\noindent
{\bf Proof.} It is clear from the definition of $\alpha_c$ that sending $c \in k^{\times}$ to $\alpha_c$  is a homomorphism; Theorem \ref{main_theo}
implies that this homomorphism is injective.
That it is an isomorphism follows from the fact that $\sha(L)^* \simeq \sha^1(k,T_{L/k})$
(see Corollary \ref {5.10}).
\bigskip

{\bf Metacyclic extensions}

\medskip

In the following we apply the main theorem to the case where $K$ is a metacyclic extension of $k$ (recall that a metacyclic extension is
a Galois extension such that all the Sylow subgroups of its Galois group are cyclic).
As before, let $X_c$ be the $k$-variety defined by the equation $(\ref{e0})$.
Assume that $K/k$ is a metacyclic extension of degree $q=\underset{j=1}{\overset{s}{\prod}} p_j^{e_j}$, where $p_j$'s are distinct primes.
Let $q_j=p_j^{e_j}$ and $r_j=q/q_j$.
For $1\leq j\leq s$, let $G_j$ be a $p_j$-Sylow subgroup of $\Gal(K/k)$ and let ${F}_{j}$ be the subfield of $K$ fixed by $G_j$.
Note that $[F_{j}:k]=r_j$.  Let $X^j_c$ be $X_c\otimes_k F_{j}$.
Then the injection $k\to F_{j}$ induces a natural injection of $X_c(k)$ to $X_c(F_{j})=X_c^j(F_{j})$.

\medskip
Suppose that $X_c$ has a $k_v$-point
%$(x_i^v)$
for all $v\in\Omega_k$.
Then $X^j_c$ has a $F_{j,w}$-point for all $w\in\Omega_{F_j}$.
Since $F_{j}$ is a cyclic extension of $k$,  we can define the Brauer-Manin map $\alpha_j$ for $X^j_c$.
The necessary and sufficient condition for the Hasse principle for $X_c$ to hold is the following :

\begin{prop}\label{metacyclic}
Assume that $K$ is a metacyclic extension. Then $X_c$ has a $k$-point if and only if $X_c$ has a $k_v$-point for all $v\in\Omega_k$
and $\alpha_j=0$ for $1\leq j\leq s$.
\end{prop}

\begin{proof} Assume that $X_c$ has a $k_v$-point for all $v\in\Omega_k$,
and
that $\alpha_j=0$ for $1\leq j\leq s$. Then  the variety $X^j_c$ has a $F_{j,w}$-point for all $w\in\Omega_{F_j}$.
%If $X_c$ has a $k_v$-point for all $v\in\Omega_k$,
%then $X^j_c$ has a $F_{j,w}$-point for all $w\in\Omega_{F_j}$.
Since $\alpha_j=0$ for all $1\leq j\leq s$, by Theorem \ref{main_theo} the variety $X_c^j$ has a $F_j$-point.
Let  $(x_{j,i})$ be a $F_j$-point of $X_c^j$, where $x_{j,i}\in (F_j\otimes_k K_i)^\times$.
Let $b_j\in\ent$ such that $\underset{j=1}{\overset{s}\sum} b_jr_j=1$, and set
$z_i=\underset{j=1}{\overset{s}{\prod}} N_{F_j\otimes K_i/K_i}(x_{j,i})^{b_j}$; then $(z_i)$ is a point of $X_c$.
%Since $\underset{j=1}{\overset{s}\sum} b_jr_j=1$, we have that  $(z_i)$ is a point of $X_c$.
The other direction is trivial.
\end{proof}

\bigskip

\section{Products of cyclic extensions}

\medskip
In this section, we suppose that $L$ is a {\it product of cyclic extensions}, and we denote
by $\sha(L)$ the obstruction group. In the following, we give a simple criterion for the vanishing
of $\sha(L)$; in other words, an easy way to decide whether the Hasse principle holds for $L$.

\medskip
 Assume that $L = \underset{i \in J}\prod K_i$,
 where $K_i/k$ is a cyclic extension of degree $d_i$.
Let $\cP$ be the set of prime numbers dividing $ \underset{i \in J}\prod  d_i$. For all $p \in \cP $ and all $i \in J$, let $K_i(p)$ be the largest subfield
of $K_i$ such that $[K_i(p):k]$ is a power of $p$, and set $L(p) = \underset{i \in J}\prod K_i(p)$.

\medskip
For any cyclic field extension $K/k$ of prime power degree, we denote by $K_{\rm prim}$  the unique subfield of $K$ of degree $p$ over $k$. Set $L(p)_{\rm prim} = \underset{i \in J} \prod K_i(p)_{\rm prim}$.

\medskip
The aim of this section is to prove the following two results :

\smallskip

\begin{theo}\label{zero}Assume that $K_{\rm prim}$ is linearly disjoint from $K_i$ for all $i\in I$. $$\sha(L) = 0 \iff  \underset{p \in \cP(L)}{\oplus}   \sha(L(p)_{\rm prim}) = 0,$$
where $\cP(L)$ is a set of prime numbers, subset of $\cP$.
\end{theo}

\medskip
The set $\cP(L)$ is determined in Theorem \ref{prime}, see below.

\medskip

\begin{theo}\label {cyclicproduct}
$$\sha (L(p)_{\rm prim}) \simeq  (\ent / p \ent)^{m_p(L)},$$
where  $m_p(L)$ is a positive integer.

\end{theo}

The value of $m_p(L)$ is given in Theorem \ref{prime}.

\bigskip We start with the proof of Theorem \ref{cyclicproduct}, which amounts to treating the case where $L$ is a product of cyclic extensions of prime degree.

\begin{theo}\label{prime}  Let $p$ be a prime number, and assume that $L$ is a product of $n$ non-isomorphic cyclic extensions of degree $p$. Then we have

\medskip
{\rm (a)} If $n \le 2$, then $\sha(L) = 0$.

\medskip
{\rm (b)} If $3 \le n \le p+1$, then either $\sha(L) = 0$, or $\sha(L) \simeq  (\ent / p \ent)^{n-2}$.

\medskip
{\rm (c)} If $n \ge p+2$, then $\sha(L) = 0$.

\end{theo}

\medskip
Note that Theorem \ref{prime} implies immediately Theorem \ref{cyclicproduct}, and gives
the value of the integer $m_p(L)$. Note also that the case $n = 1$ is a special case Hasse's cyclic norm theorem, the case $n = 2$ follows from H\"urlimann's result \cite{H} Prop. 3.3, (see also Proposition \ref{hurlimann}), and that the case $n = 3$, $p = 3$ is a result
of Colliot-Th\'el\`ene, cf. \cite{CT}, Th\'eor\`eme 4.1.

\medskip

In order to prove Theorem \ref{prime}, we need to come back to the definition of $\sha(L) = \sha(K,K')$ in the case where $K$ is cyclic
of prime degree, and give a description of this group in terms of partitions.

\medskip
We keep the notation of 5.1, with $e = 1$. In particular, $p$ is a prime number, and $L = K \times K'$, where $K$ is a cyclic extension of $k$ of degree $p$.
Recall that $E_i = K \otimes K_i$, and note that $E_i$ is either a cyclic field extension of $K_i$ or a product of $p$ copies of $K_i$. Let $J$ be
the subset of $i \in \cI$ such that $E_i/K_i$ is a field extension, and let $r = |J|$.

\medskip
Recall that $\Sigma_i$ is the set of $v \in \Omega_k$
such that $E_i^v$ is the product of $p$ copies of $K_i^v$. For all $J' \subset J$ with $J' \not = J$, set
$\Omega(J') = \underset{i \not \in J'} \cap \Sigma_i$, and let $\Omega(J) = \Omega_k$.  By lemma \ref{primedegreepartition}, the group
$G(K,K')$ is in bijection with the set of partitions $(J_0,\dots,J_{p-1})$ of $J$ such that $ \underset{n \in \ent / p \ent} \cup \Omega(J_n) = \Omega_k.$
We identify $G(K,K')$ with the set of these partitions. Note that under this identification, $D(K,K')$ corresponds to the partitions
where one of the subsets is $J$, and all the others are empty; these will be called the trivial partitions of $J$.

\medskip

For all $n \in \ent /p \ent$ and all $a \in (\ent / p \ent)^r$, set $J_n(a) = \{ i \in J \ | \ a_i = n \}$. Then lemma \ref{primedegreepartition}  can be
reformulated as follows :

\begin{lemma}\label{part} $G(K,K')$ is in bijection with the set $$\{a \in (\ent / p \ent)^r \ | \  \underset{n \in \ent / p \ent} \cup \Omega(J_n(a)) = \Omega_k\}.$$

\end{lemma}

\bigskip

\bigskip
{\bf Proof of Theorem \ref{prime}}

\bigskip

Note first that (a) follows from Proposition \ref{hurlimann}. From now on, we assume that
$n \ge 3$.
Theorem \ref{prime},  as well as a precise condition for when $\sha(L) = 0$ in case {\rm (b)}, is a consequence of Proposition \ref{F} below.

\medskip
For any positive integer $d$, a finite separable extension $F$ of $k$ is said to have {\it local degrees $\le d$} if for all places $v\in\Omega_k$,
the \'etale algebra $F\otimes_k k_v$ is a product of field extensions of  $k_v$ with degrees $\le d$.

\begin{prop}\label{F}  Let $p$ be a prime number, and assume that $L$ is a product of distinct field extensions of degree $p$ of $k$, at least one of which is cyclic.

\medskip
Then
$\sha(L) \not = 0$ $\iff$ the factors of $L$ are distinct subfields of a field extension $F/k$ of degree $p^2$, and all the local degrees of
$F$ are $\le p$.

\medskip
Moreover, if $\sha(L) \not = 0$, and if $L$ is a product of $n$ distinct degree $p$ field extensions of $k$, then $\sha(L) \simeq  (\ent / p \ent)^{n-2}$.

\end{prop}

\medskip
\noindent
{\bf Proof.} Let $K$ be a cyclic factor of $L$, and let us write $L = K \times K'$, where $K'$ is a product of field extensions of degree $p$ of $k$.
Suppose that $\sha(L) \not = 0$. Then there exists a partition $(I_0,I_1)$ of $J$ such that $\Omega(I_0) \cup \Omega(I_1) = \Omega_k$.
Indeed, let $(J_0,\dots,J_{p-1})$ be a non-trivial partition of $J$ such that $\underset{r \in \ent/p\ent} \cup \Omega(J_i) = \Omega_k$. Without
loss of generality, we can assume that $J_0$ is not empty.
Set $I_0 = J_0$, and let $I_1 = \underset  {i \not = 0} \cup J_i$; then we have $\Omega(I_0) = \Omega(J_0)$, and $\Omega (J_r) \subset \Omega(I_1)$
for all $r \not = 0$. Therefore $\Omega(I_0) \cup \Omega(I_1) = \Omega_k$, as claimed. Let $K_i$ and $K_j$ be two distinct factors of $K'$, and
let $K_i K_j$ be the composite of $K_i$ and $K_j$. For all $v \in \Sigma_i$, we have

$$K \otimes _k (K_i K_j)^v \simeq K \otimes _k  K_i^v \otimes_{K_i^v} (K_i K_j)^v ,$$

\medskip
\noindent
and, since $v \in \Sigma_i$,  this is isomorphic to the product of $p$ copies of $(K_iK_j)^v$.

\medskip
Let $i \in I_0$ and $j \in I_1$.
As we have $\Omega(I_0) \cup \Omega(I_1) = \Omega_k$, the tensor product $K \otimes _k (K_i K_j)^v$ is isomorphic
to the product of $p$ copies of $(K_iK_j)^v$ for all $v \in \Omega_k$. This implies that $K$ is a subfield of $K_i K_j$.
Recall that $K$ is cyclic, and that $K_i$, $K_j$ are not isomorphic; hence we have $K \otimes_k K_i \simeq K K_i \subset K_i K_j$. The degree of $K_iK_j$ is at most $p^2$, hence
we have $KK_i = K_iK_j = K K_j$, and $K_i \otimes_k K_j \simeq K_i K_j$ is of degree $p^2$ over $k$.

\medskip Let $i \in I_0$, and set $F = K K_i$; we just saw that $F$ is independent of the choice of $i$, and that $F = K_i K_j$ for all $j \in I_1$.  This
shows that $K_i$ is a subfield of $F$ for all $i \in J$. Since $(I_0,I_1)$ represents a non-trivial element of $\sha(L)$, for all $v \in \Omega_k$ there
exists $i \in J$ such that $F^v \simeq K \otimes_k K_i^v$ is isomorphic to a product of $p$ copies of $K_i^v$. Therefore all the local degrees of $F$ are $\le p$.

\medskip
Conversely, let $F$ be a separable extension of degree  $p^2$ of $k$ such that all the factors of $L$ are distinct subfields of $F$. It suffices to
prove that all non-trivial partitions
$(J_0,\dots,J_{p-1})$ of $J$ satisfy $\underset{r \in \ent/p\ent} \cup \Omega(J_i) = \Omega_k$. Suppose that this is not the case. Let $(J_0,\dots,J_{p-1})$
be a non-trivial partition of $J$ with $\underset{r \in \ent/p\ent} \cup \Omega(J_i) \not = \Omega_k$. Let $v \in \Omega_k$ with
$v \not \in \underset{r \in \ent/p\ent} \cup \Omega(J_i)$. Since
$v \not \in \Omega(J_0)$, there exists $i \not \in J_0$ such that $iv\not \in \Sigma_i$. Let $r \in \ent / p \ent$ such that $i \in J_r$; since $v \not \in \Omega(J_r)$,
there exists $j \not  \in J_r$ such that $v \not \in \Sigma_j$.

\medskip
Since the degree $p$ extensions $K$, $K_i$ and $K_j$ are distinct subfields of $F$, we have $F \simeq K \otimes K_i \simeq K \otimes K_j$.
Note that $[K^v:k_v] = p$, because $v \not \in \Sigma_i$. Let us write $K_i^v$ as a product of separable extensions of $k_v$. If one of the factors $M_s$ of $K_i^v$ is such that $1 < [M_s:k_v] < p$, then $M_s$ and $K^v$ are linearly disjoint, and this contradicts the assumption
that all the local degrees of $F$ are $\le p$. Hence $K_i^v$ is either a degree $p$ field extension of $k_v$, or
a product of $p$ copies of $k_v$. However, if $K_i^v$ and $K^v$ are both fields, then $E_i^v$ is a field extension of degree $p^2$ of $k_v$. Since
$F^v \simeq E_i^v$, this contradicts the hypothesis that all the local degrees of $F$ are $\le p$. Therefore $K_i^v$ is a product of $p$ copies of $k_v$,
and hence $F^v \simeq E_i^v$ is a product of $p$ copies of $K^v$.

\medskip Set $d = [K_iK_j:k]$.
Since $v \not \in \Sigma_j$, the same argument shows that $K_j^v$ is a product of $p$ copies of $k_v$, hence $(K_iK_j)^v$ is a product of $d$ copies of $k_v$.  Note that $(K_iK_j)^v$ is a subalgebra of $F^v$, and that $F^v$ is a product of $p$ copies of $K^v$; hence we have $d \le p$. As $K_i$ and $K_j$ are
distinct subfields of $K_iK_j$, we have $d = r p$ for some integer $r > 1$, and this leads to a contradiction.

\medskip
Hence for all non-trivial partitions $(J_0,\dots,J_{p-1})$ of $J$ we have $\underset{r \in \ent/p\ent} \cup \Omega(J_i) = \Omega_k$. This shows
that $\sha(K,K') = \sha(L) \simeq (\ent / p \ent)^{n-2}$.

\bigskip
\centerline {\bf Proof of Theorem \ref{zero}}

\medskip
Assume now that $L$ is a product of $n$ cyclic extensions, $L = K_1 \times \dots \times K_n$, where $K_i/k$ is a cyclic extension of degree $d_i$, and let
$J = \{1,\dots,n\}$. Note
that $\sha(L) = \sha(K_i,K_i')$ for any $i \in J$, where $L = K_i \times K'_i$. This will be used repeatedly in the sequel.

\medskip
Let $\cP$ be the set of prime numbers dividing $d_1 \dots d_n$. For all $p \in \cP $ and all $i \in J$, let $K_i(p)$ be the largest subfield
of $K_i$ such that $[K_i(p):k]$ is a power of $p$, and set $L(p) = K_1(p) \times \dots \times K_n(p)$.

\begin{prop} \label{prime decomposition} We have
$$\sha (L) =  \underset{p \in \cP(L)}{\oplus}  \sha(L(p)).$$

\end{prop}

\noindent
{\bf Proof.} This follows from Proposition \ref {sha_isom}, and from the fact that $L$ is a product
of cyclic extensions.

\medskip
\begin{lemma}\label{injectivity} Let $p$ be a prime number, and let $K_i/k$, $i \in J$, be cyclic extensions of degree a power of $p$ of $k$.
For all $i \in J$, let $N_i/k$ be a subextension of $K_i/k$. Then
$\sha(\underset{i \in J} \prod N_i)$ injects into $\sha(\underset{i \in J} \prod K_i)$.

\end{lemma}

\medskip
\noindent
{\bf Proof.} This follows from Proposition \ref{sha 0 to sha}, and Remark \ref {injective} following this
proposition.

\bigskip
\noindent
{\bf Proof of Theorem \ref{zero} } Assume that $\sha(L) = 0$. By Lemma \ref{injectivity} and
Proposition \ref{prime decomposition}, the group
$\sha(L_{\rm prim})$ injects into $\sha(L)$, hence this implies that $\sha(L_{\rm prim}) = 0$.
Conversely, suppose that $\sha(L_{\rm prim}) = 0$. By Proposition \ref{prime decomposition}, we
may assume that $L$ is a product of extensions of degree a power of a prime $p$.
Let us write $L = K \times K'$, for some cyclic field extension $K/k$; then $L_{\rm prim} =
K_{\rm prim} \times K'_{\rm prim}$.  Since $\sha(L_{\rm prim}) = 0$, by Proposition \ref{prim prim}
we have $\sha(K,K'_{\rm prim}) = 0$.
Permuting $K$ with one of the other cyclic factors and repeating the same procedure, we obtain
$\sha(L) = 0$.

\bigskip

\begin{example}
Let $p$ be a prime number, and let $F/k$ be an extension with Galois group
$C_p \times C_p$,
where $C_p$ denotes the
cyclic group of order $p$. Let $K_1,\dots,K_{p+1}$ be the distinct subfields of degree $p$ of $F$. Set $L = K_1 \times \dots \times K_{p+1}$. Then
by Proposition \ref{F}, we have $\sha(L) = 0$ or $\sha(L) = (\ent /p \ent)^{p-1}$. Moreover, we have

\medskip
\centerline {$\sha(L) = 0 \iff$ there exists $v \in \Omega_k$ such that $F^v$ is a field.}

\medskip

\noindent
$\bullet$
Assume first that there exists $v \in \Omega_k$ such that $F^v$ is a field. Then $\sha(L) = 0$, hence for all $c \in k^{\times}$, we have $X_c (k) \not
= \emptyset$. In other words, we have $$N_{L/k}(L^{\times}) = k^{\times}$$ in this case.

\medskip

\noindent
$\bullet$ Assume now that all the local degrees of $F$ are $\le p$. Then by Proposition \ref{F} we have $\sha(L) = (\ent /p \ent)^{p-1}$.

\medskip
Let $\Omega_i$ be the set of $v \in \Omega_k$ such that $K_i^v$ is split. Note that we have $\Omega_1 \cup \dots \cup \Omega_{p+1} = \Omega_k$.
This implies that $X_c(k_v) \not = \emptyset$ for all $v \in \Omega_k$ and for all $c \in k{\times}$.

\medskip
Set $K = K_{p+1}$.
For all $c \in k^{\times}$ and for all $v \in \Omega_k$, let us denote by $[K,c]_v \in \ent/p \ent$ the image of  ${\rm inv}(K,c)_v$ by the isomorphism
${1 \over p} \ent /\ent \simeq p \ent / \ent$. Then the map

$$f : k^{\times} / N_{L/k}(L^{\times}) \to (\ent / p \ent)^{p-1}$$ given by
$$c \mapsto (\underset{\Omega_1}\sum [K,c]_v, \dots, \underset{\Omega_{p-1}}\sum [K,c]_v),$$ is an isomorphism.

\bigskip
When $p = 2$, we recover a well-known result of Serre and Tate, see [CF 67], Exercise 5.2, page 360; see also [CT 14], Proposition 5.1.

\end{example}

\bigskip
\bigskip
Eva Bayer--Fluckiger

EPFL-FSB-MATHGEOM-CSAG

Station 8

1015 Lausanne, Switzerland

\medskip

eva.bayer@epfl.ch

\bigskip
Tingyu Lee

EPFL-FSB-MATHGEOM-CSAG

Station 8

1015 Lausanne, Switzerland

\smallskip

and

\smallskip

Technische Universit\"at Dortmund

Fakult\"at f\"ur Mathematik, Lehrstuhl LSVI

Vogelpothsweg 87, 44227 Dortmund, Germany

\medskip
tingyu.lee@gmail.com

\bigskip

Raman Parimala

Department of Mathematics $ \&$ Computer Science

Emory University

Atlanta, GA 30322, USA.

\medskip
parimala@mathcs.emory.edu

\end{document}